\documentclass{article}

\usepackage{graphicx}%
\usepackage{multirow}%
\usepackage{amsmath,amssymb,amsfonts}%
\usepackage{amsthm}%
\usepackage{mathrsfs}%
\usepackage[title]{appendix}%
\usepackage{xcolor}%
\usepackage{textcomp}%
\usepackage{manyfoot}%
\usepackage{booktabs}%
\usepackage{algorithm}%
\usepackage{algorithmicx}%
\usepackage{algpseudocode}%
\usepackage{hyperref}
\usepackage{listings}%
\usepackage{graphicx}  
\usepackage{multirow}
\usepackage{booktabs}\let\cline\cmidrule
\usepackage{url}
\usepackage{geometry}
\usepackage{wrapfig,lipsum,booktabs}
\allowdisplaybreaks

\newtheorem{theorem}{Theorem}
\newtheorem{proposition}[theorem]{Proposition}%
\newtheorem{corollary}{Corollary}[theorem]
\newtheorem{problem}{Problem}[section]
\newtheorem{lemma}[theorem]{Lemma}        

\newtheorem{definition}{Definition}%

\title{Higher-Order Root-Finding Algorithm and its Applications}

\author{Wei Guo Foo and Chik How Tan}
\date{Temasek Laboratories, National University of Singapore\\
	5A Engineering Drive 1, \#09-02, Singapore 117411, Singapore.
	$\{\text{fwg, tsltch}\}$@nus.edu.sg}

\begin{document}


	\maketitle

\abstract{Root-finding method is an iterative process that constructs a sequence converging to a solution of an equation. Householder's method is a higher-order method that requires higher order derivatives of the reciprocal of a function and has disadvantages. Firstly, symbolic computations can take a long time, and numerical methods to differentiate a function can accumulate errors. Secondly, the convergence factor existing in the literature is a rough estimate. In this paper, we propose a higher-order root-finding method using only Taylor expansion of a function. It has lower computational complexity with explicit convergence factor, and can be used to numerically implement Householder's method. As an application, we apply the proposed method to compute pre-images of $q$-ary entropy functions, commonly seen in coding theory. Finally, we study  basins of attraction using the proposed method and compare them with other root-finding methods.\\

\noindent {\bf Keywords}: Higher-order root-finding methods, {\textit q}-ary entropy functions, Householder's method.\\
\noindent {\bf MSC Classification}: 65H05, 94-08}



\section{Introduction}\label{sec1}
Finding a solution to an equation $f(x)=0$ may require numerical methods, especially when a closed-form formula can not be found. It requires setting up iterative schemes to construct sequences based on initial guesses converging to the nearest solution, and their convergence speeds vary according to different methods. The well-known ones are derivative-based methods such as Newton's and Halley's method, with respective convergence orders of 2 and 3. Householder's method \cite{Householder-1970,Pascal-Xavier-21,Breuer-Zwas-1984} (or in some places known as the K\"onig's method \cite{ABD-1997, Buff-Henriksen-2003, McNamee-2007}) is one of their generalisations which achieve higher orders of convergence. However, its use has a few limitations. Firstly, its implementation requires taking derivatives of the reciprocal of the function, which for high orders, can be symbolically difficult to implement. There are workarounds using finite difference methods to find numerical approximations to the derivative of a function but they accumulate errors and can lead to inaccurate results  \cite{GRS-2007, Krasny-2025}. Secondly, existing literature describes the convergence rate using the big O notation. Hence the method may fail to produce a converging sequence despite its implementation for high orders since the convergence factor may be very large, underscoring the importance of finding an accurate estimate of its value. This paper is therefore motivated by the following problem: 
\medskip
\begin{problem}
Given an equation $f(x)=0$ with $\alpha$ being a solution. Is there an effective and efficient numerical root-finding method that constructs a sequence of the form:
\begin{equation}\label{probeq}
x_{n+1}=x_{n}+P(x_{n}),
\end{equation}
with convergence to the actual solution in the following way for high orders of $k$:
\begin{equation}\label{probconv}
|x_{n+1}-\alpha|\leq C|x_{n}-\alpha|^{k+1},
\end{equation}
and with an accurate estimation of convergence factor $C$?
\end{problem}
\medskip
To address this problem, we propose a method which requires a set of equations consisting of Taylor series expansion of $f$ centered at the $n$-th iterated value $x_{n}$ evaluated at the nearest solution $\alpha$ treated as a symbolic variable that remained to be solved. Then we eliminate some Taylor coefficients by using linear combinations of equations similar to Gaussian elimination to obtain equation \eqref{probeq}. In addition, the elimination process  allows us to compute $C$ in equation \eqref{probconv}. The following theorem describes our main result of the paper. 
\medskip

\begin{theorem}\label{maintheorem}
	Let $f(x)$ be a $k+1$ continuously-differentiable function and $\alpha$ be one of the solutions to the equation $f(x)=0$. Define iteratively:
	\begin{equation}\label{mtAk}
		\begin{aligned}
			A_{0} &= 1,\quad A_{1}=f',\quad A_{2}=(f')^{2}-\frac{1}{2}f''f,\\
			A_{j+1} &= f'A_{j}-f\widehat{A}_{j},\quad \text{where}\quad  \widehat{A}_{j} = \sum_{l=2}^{j+1}\frac{1}{l!}(-f)^{l-2}f^{(l)}A_{j+1-l}.
		\end{aligned}
	\end{equation} 
	Let $x_{0}$ be the initial guess and consider the following sequence:
	\[
	x_{n+1} = x_{n} - \frac{f(x_{n})A_{k-1}(x_{n})}{A_{k}(x_{n})}.
	\]
	Then there exists an expression $C$ such that 
	\[
	\left|x_{n+1}-\alpha\right|\leq C\left|\alpha-x_{n}\right|^{k+1}.
	\]
	\end{theorem}

Recall that the Householder's method of order $k$ for the equation $f(x)=0$ with an initial guess $x_{0}$ to one of the roots $\alpha$ is a sequence given by the iteration:
\begin{equation}
	x_{n+1}= x_{n}+k\frac{(1/f)^{(k-1)}(x_{n})}{(1/f)^{(k)}(x_{n})},
\end{equation}
where $(1/f)^{(k)}$ denotes the $k$-th derivative of the reciprocal of the function $f$. As a consequence of Theorem \ref{maintheorem}, we show that our proposed method can be used to numerically implement Householder's method, which is our second main result.
\medskip
\begin{theorem}\label{intro-Theorem-Householder}
	Let $A_{k}$ be coefficients given by equation \eqref{mtAk}.  Then
	\[
	k\frac{(1/f)^{(k-1)}}{(1/f)^{(k)}} = -f\frac{A_{k-1}}{A_{k}}.
	\]
	In other words, the Householder's method can numerically be done by Theorem \ref{maintheorem}.
\end{theorem}
\medskip
Due to Theorem \ref{intro-Theorem-Householder}, we are able to obtain a more accurate description of the convergence factor for the Householder's method.
\medskip
\begin{corollary}\label{maintheorem-cor}
	Let $f(x)=0$ be an equation and $x_{0}$ be the first guess of the actual solution $\alpha$. Let $x_{n}'$ and $x_{n}''$ be sequences given  by
	\begin{equation}
		\begin{aligned}
			x_{n+1}' &= x_{n}' + k\frac{(1/f)^{(k-1)}(x_{n}')}{(1/f)^{(k)}(x_{n}')},\\
			x_{n+1}'' &= x_{n}'' -\frac{f(x_{n}'')A_{k-1}(x_{n}'')}{A_{k}(x_{n}'')}.
		\end{aligned}
	\end{equation}
	Let $C$ be such that 
	\[
	|x_{n+1}''-\alpha|\leq C(x_{n}'',\alpha)|x_{n}''-\alpha|^{k+1}.
	\]
	Then for the same $C$,
	\[
	|x_{n+1}'-\alpha|\leq C(x_{n}'',\alpha)|x_{n}'-\alpha|^{k+1}.
	\]
\end{corollary}
\medskip

There are several advantages of our proposed root-finding method in Theorem \ref{maintheorem}. Firstly, it has lower computational complexity. If $k$ is the order of convergence, an estimation shows that it takes about $O(k^{3})$ steps to compute the sequence $\{A_{0},\cdots,A_{k}\}$. In contrast, the time complexity of Householder's method can grow exponentially with respect to $k$ if  $\frac{d^{k}}{dx^{k}}\frac{1}{f}$ were to be symbolically calculated. Secondly, it allows Householder's method to be numerically implemented with lower computational complexity by avoiding symbolic calculations. Lastly, it gives us better clarity of the time complexity and a more explicit convergence factor of the Householder's method.

As a demonstration of our proposed root-finding method, we apply Theorem \ref{maintheorem} to  find the pre-images of the $q$-ary entropy function, which plays an important role in information theory and cryptography. It describes how much choice is involved in the selection of events and how uncertain we are of the outcome. It was introduced by Shannon \cite{Shannon1948, MacKay-2003} in 1948, and is mathematically written as:
\[
H = -\sum_{i=1}^{n}p_{i}\log_{b}(p_{i}),
\]
where $p_{i}$'s are the probabilities of occurences of events so that their sum is $1$. The case $n=b=2$ is the well-known binary entropy function:
\begin{equation}\label{H2}
H_{2}(x)=-x\log_{2}x - (1-x)\log_{2}(1-x).
\end{equation}
More generally when dealing with finite fields with $q$ elements, one has the $q$-ary entropy function:
\medskip
\begin{definition}\label{Hq}
	For any $q\geq 2$, the $q$-ary entropy function $H_{q}:[0,1]\longrightarrow \mathbb{R}$ is defined as 
	\[
	H_{q}(x) :=x\log_{q}(q-1)-x\log_{q}x-(1-x)\log_{q}(1-x).
	\]
\end{definition}

Recall that a linear code is a vector subspace of $\mathbb{F}_{q}^{n}$, and suppose its dimension grows proportionally to $n$ at a certain rate $0<R<1$. A successful decoding of a syndrome depends on finding the minimal distance of the linear code, which is an NP-complete problem \cite{Vardy-2002} especially for large $n$. Nonetheless, an estimation of the minimal distance of a general random linear code is possible by the Gilbert-Varshamov distance \cite{Gilbert-1952, Varshamov-1957}, denoted by $D_{\text{GV}}(R)$, which  is the unique solution $x\in [0,\frac{q-1}{q}]$ to
\begin{equation}\label{intro-eq1}
H_{q}(x)=1-R.
\end{equation}
A classical result shows that in almost all of the linear codes, the distance between any two elements is at least $\lfloor D_{\text{GR}}(R)n\rfloor$. Thus computing the minimum distance of a linear code requires finding  solution to equation \eqref{intro-eq1}. To date, other than the obvious ones $1-R\in \{0,1\}$, there is no known closed-form solution.  Hence the second problem we aim to address in this paper is the following:
\medskip
\begin{problem}\label{prob2}
For any $y\in [0,1]$, find the numerical solution to $H_{q}(x)=y$ 
for $x\in [0,\frac{q-1}{q}]$.
\end{problem}
\medskip
A common challenge to implementing root-finding methods is the choice of first guess $x_{0}$ of the solution of the equation $f(x)=0$. In the context of Problem \ref{prob2} , this can be overcome by replacing the $q$-ary entropy function with  simpler approximating functions, and solving the new equations instead. Intuitively, due to the closeness between these functions, the first guess should not be too far away from the actual solutions.  We will study various ways to approximate the $q$-ary entropy function, and compare the performances of root-finding methods with these guesses. Numerical results show that our proposed root-finding method runs faster than the classical Householder's method.

As a third application of Theorem \ref{maintheorem}, we apply the root-finding method to complex equations with complex solutions and examine their basins of attraction, commonly studied in dynamical systems. It describes the set of starting points giving rise to sequences which converge to the same solution. As a consequence of Theorem \ref{intro-Theorem-Householder}, we show that both Theorem \ref{maintheorem} and Householder's method have the same basins of attractions.

This paper is organised as follows. Section 2 describes the root-finding method and provides a proof to Theorem \ref{maintheorem}. Section 3 discusses how Theorem \ref{maintheorem} can be used to numerically implement Householder's method. Section 4 applies the discussion to find the pre-images of the $q$-ary entropy functions.  Finally in Section 5, we show that both Theorem \ref{maintheorem} and Householder's method have the same basins of attraction, and provide some illustrations.

\section{Proposed Higher Order Root Finding Scheme}\label{section2}

We will solve for $x$ in $f(x)=y$ for a given fixed $y$. By rewriting the equation as $f(x)-y=0$, and reassigning the left-hand side as our new function $f(x)$, we may without loss of generality reduce the problem to finding the solution for the equation $f(x)=0$. Our starting point is to expand $f(x)$ around our $n$-th iterated estimated root $x_{n}$ in terms of Taylor series up to order $k+1$:
\begin{equation*}
\begin{aligned}
f(x) &= f(x_{n})+f'(x_{n})(x-x_{n})+\cdots  + \frac{1}{k!}f^{(k)}(x_{n})(x-x_{n})^{k}+\frac{1}{(k+1)!}f^{(k+1)}(\xi_{n})(x-x_{n})^{k+1},
\end{aligned}
\end{equation*}
where by Taylor's theorem, $\xi_{n}$ lies between $x$ and $x_{n}$. Thus $\xi_{n}$ depends on $x$. If $\alpha$ is a solution to $f(x)=0$, then we have the fundamental equation:
\begin{equation*}
\begin{aligned}
0 &= f(\alpha)\approx \sum_{j=0}^{k}\frac{1}{j!}f^{(j)}(x_{n})(\alpha-x_{n})^{j}+\frac{1}{(k+1)!}f^{(k+1)}(\xi_{n})(\alpha-x_{n})^{k+1}. 
\end{aligned}
\end{equation*}

\subsection{The Root-Finding Method and Proof of Theorem \ref{maintheorem}}\label{algo-subsect}

Our algorithm consists of setting up a system of $k$ equations consisting of Taylor's expansion of $f$ up to a fixed order in the following manner:

\begin{equation}\label{Taylorpoly}
	\begin{aligned}
		0 &= f_{2}(\alpha) := f(x_{n})+f'(x_{n})(\alpha-x_{n})+\frac{1}{2!}f''(\xi_{2})(\alpha-x_{n})^{2},\\
		&\vdots\\
		0 &= f_{k+1}(\alpha) := f(x_{n})+f'(x_{n})(\alpha-x_{n})+\cdots+\frac{1}{k!}f^{(k+1)}(\xi_{k+1})(\alpha-x_{n})^{k+1}.
	\end{aligned}
\end{equation}

\subsubsection{Outline}\label{outline}
 We first outline the idea for the proof of Theorem \ref{maintheorem}. Consider:
\begin{equation}\label{Taylorpoly2}
	\begin{aligned}
	 0 &= F_{j+2,0}^{j} + F_{j+2,1}^{j}\cdot (\alpha-x_{n}) + R_{j+2,j+2}^{j}\cdot (\alpha-x_{n})^{j+2},\\
	 &\vdots\\
	 0 &= F_{l,0}^{j} + F_{l,1}^{j}\cdot (\alpha-x_{n}) + F_{l,2}^{j}\cdot (\alpha-x_{n})^{2} + \cdots + F_{l, l-j-1}^{j}\cdot(\alpha-x_{n})^{l-j-1}\\
	 &\hspace{8cm} +R_{l,l}^{j}\cdot (\alpha-x_{n})^{l},\\
	 &\vdots\\
	 0 &= F_{k+1,0}^{j} + F_{k+1,1}^{j}\cdot (\alpha-x_{n}) + F_{k+1,2}^{j}\cdot (\alpha-x_{n})^{2} +  \cdots + F_{k+1,k-j}^{j}\cdot (\alpha-x_{n})^{k-j}\\
	 &\hspace{8cm} +R_{k+1,k+1}^{j}\cdot (\alpha-x_{n})^{k+1},
	\end{aligned}
\end{equation}
where we let $F_{l,r}^{j}$  denote the coefficients of $(\alpha-x_{n})^{r}$ at the $j$-th step. The case $j=0$ corresponds to equation \eqref{Taylorpoly}.  At the $j$-th step, we use the first equation to eliminate the term containing $(\alpha-x_{n})^{l-j-1}$ in the $l$-th line:
\begin{equation}\label{steps}
	\begin{aligned}
		0 &= -F_{l,l-j-1}^{j}\cdot (\alpha-x_{n})^{l-j-2}\left(F_{j+2,0}^{j} + F_{j+2,1}^{j}\cdot (\alpha-x_{n}) + R_{j+2,j+2}^{j}\cdot (\alpha-x_{n})^{j+2}\right)\\
		&
		+F_{j+2,1}^{j}\cdot \bigg(F_{l,0}^{j} + F_{l,1}^{j}\cdot (\alpha-x_{n}) + F_{l,2}^{j}\cdot (\alpha-x_{n})^{2} + \cdots + F_{l, l-j-1}^{j}\cdot(\alpha-x_{n})^{l-j-1}\\
		&\hspace{8cm} +R_{l,l}^{j}\cdot (\alpha-x_{n})^{l}\bigg)\\
		&:= F_{l,0}^{j+1} + F_{l,1}^{j+1}\cdot (\alpha-x_{n}) + \cdots + F_{l,l-j-2}^{j+1}\cdot (\alpha-x_{n})^{l-j-2}+
		R_{l,l}^{j+1}\cdot (\alpha-x_{n})^{l},
	\end{aligned}
\end{equation}
which gives $F_{l,r}^{j+1}$ coefficients of $(\alpha-x_{n})^{r}$ at the $(j+1)$-th step. We then remove the first line, and repeat the process. Continuing this way, we will obtain equation of the form:
\[
0 = F_{k+1,0}^{k-1} + F_{k+1,1}^{k-1}(\alpha-x_{n})+R_{k+1,k+1}^{k-1}\cdot (\alpha-x_{n})^{k+1}.
\]

 For the rest of the subsection, we will treat the $f_{i}(\alpha)$'s in equation \eqref{Taylorpoly} as functions in the variable $\alpha$. Having set up the context, we will prove Theorem \ref{maintheorem}.

\subsubsection{Proof of Theorem 1}
\begin{proof}[Proof of Theorem \ref{maintheorem}]
Given a sequence of functions in variable $\alpha$:
\[
  \overrightarrow{P_{l}} := (P_{l}(\alpha),\dots,P_{k+1}(\alpha)),
\]
we define the operator $\text{elim}$ that sends each element of $\overrightarrow{P_{l}}$ to 
\begin{align*}
  P_{l} &\longmapsto P_{l},\\
  P_{j} &\longmapsto \left(P_{j}(\alpha)P_{l}'(\alpha)\big|_{\alpha=x_{n}}-\frac{1}{(j-l+1)!}P_{l}(\alpha)P_{j}^{(j-l+1)}(\alpha)\big|_{\alpha=x_{n}}(\alpha-x_{n})^{j-l}\right),
\end{align*}
where the differentiation is done with respect to the variable $\alpha$. With $f_{2},\, \cdots,\ f_{k+1}$ given in equation \eqref{Taylorpoly}, successively we write
\begin{align*}
  \overrightarrow{r_{2}} := (f_{2}(\alpha),\cdots,f_{k+1}(\alpha)),\qquad
  \overrightarrow{r_{j+1}} := \pi_{j}\big(\text{elim}(\overrightarrow{r_{j}})\big),
\end{align*}
where $\pi_{i}:(x_{i},x_{i+1},\dots,x_{k+1})\longmapsto (x_{i+1},\cdots,x_{k+1})$ is the projection of a vector onto the last $k-i+1$ elements. The final  $\overrightarrow{r_{k+1}}$ therefore consists only of one element, which is a polynomial of the form 
\begin{eqnarray}\label{rf-expr}
	\begin{aligned}
  r_{k+1,0}(x_{n})+r_{k+1,1}(x_{n})(\alpha-x_{n})+r_{k+1,k+1}(x_{n},\xi_{i})(\alpha-x_{n})^{k+1}.
  \end{aligned}
\end{eqnarray}
Equating this to zero, the equation can be rewritten as 
\[
  \alpha-\left(x_{n}-\frac{r_{k+1,0}(x_{n})}{r_{k+1,1}(x_{n})}\right) = -\frac{r_{k+1,k+1}(x_{n},\xi_{i})}{r_{k+1,1}(x_{n})}(\alpha-x_{n})^{k+1}=: C_{k+1}\cdot (\alpha-x_{n})^{k+1}.
\]
If we recursively let
\[
  x_{n+1} := x_{n}-\frac{r_{k+1,0}(x_{n})}{r_{k+1,1}(x_{n})},
\]
then the sequence $(x_{n})$ satisfies the convergence speed
\begin{equation}\label{ocr}
  \alpha-x_{n+1}=-\frac{r_{k+1,k+1}(x_{n},\xi_{i})}{r_{k+1,1}(x_{n})}(\alpha-x_{n})^{k+1}.
\end{equation}

Following the procedure described above, we can track the changes of the coefficients at each step such as in the previous subsection \ref{outline}, and finally obtain explicit expressions of $r_{k+1,0}(x_{n})$ and $r_{k+1,1}(x_{n})$ below in equation \eqref{propmethod-eq}:
\begin{equation} \label{propmethod-eq}
   \begin{aligned} 
    0 &= f\cdot\left( (f')^{2^{k-3}}(A_{2})^{2^{k-4}}\cdots (A_{r})^{2^{k-r-2}}\cdots A_{k-2} \right)A_{k-1}\\
    &\hspace{0.5cm} +(f')^{2^{k-3}}(A_{2})^{2^{k-4}}\cdots (A_{r})^{2^{k-r-2}}\cdots A_{k-2}A_{k}\cdot (\alpha-x_{n}) + D_{k+1}\cdot(\alpha-x_{n}) ^{k+1},
  \end{aligned}
\end{equation}
where the coefficients $A_{k}$ are given by the recurrence relations:
\begin{equation}\label{Ak}
\begin{aligned}
  A_{0} &= 1,\quad A_{1}=f',\quad A_{2}=(f')^{2}-\frac{1}{2}f''f,\\
  A_{j+1} &= f'A_{j}-f\widehat{A}_{j},\quad \text{where}\quad  \widehat{A}_{j} = \sum_{l=2}^{j+1}\frac{1}{l!}(-f)^{l-2}f^{(l)}A_{j+1-l}.
\end{aligned}
\end{equation} 
The root-finding algorithm is therefore:
\begin{equation}\label{Ak2}
  x_{n+1}=x_{n}-\frac{fA_{k-1}}{A_{k}}.
\end{equation}
which is the result in Theorem \ref{maintheorem}. Combined with equation \eqref{ocr}, we have shown that equation \eqref{Ak2} has $(k+1)$-th order of convergence. Moreover, we can also track the expression of $D_{k+1}$ in the similar way as we have done for $A_{k}$ and obtain an explicit convergence factor $C_{k+1}$. This concludes the proof of Theorem \ref{maintheorem}. 
\end{proof}


Algorithm \ref{RFA} shows the pseudo-code for its numerical implementation, and we discuss its time complexity as a function of the order of convergence $k$. For simplicity we assume $\overrightarrow{\bf f}$ is known, and we count the number of arithmetic operations to obtain $A_{k+1}$. For each $s$, it takes about $s^2+6s+2$ steps to compute $A_{s+1}$. Since obtaining $A_{k+1}$ requires the previous values $A_{3}$, $\cdots$, $A_{k}$, computing the final term $A_{k+1}$ therefore needs approximately $\sum_{s=2}^{k-1}(s^{2}+6s+2)=\frac{1}{3}k^{3}+\frac{5}{2}k^{2}+\frac{11}{6}k-12=O(k^{3})$ steps. However, the overall time complexity also depends on computing $\overrightarrow{\bf f}$, and therefore on the choice of function $f$. 

On the other hand, Householder's method relies heavily on computing the derivatives of $\frac{1}{f}$. If we see $f$ as a symbolic function, then the number of terms in the expanded numerator of $\frac{d^{k}}{dx^{k}}\frac{1}{f}$ grows at least $e^{0.24k}$ for $2\leq k\leq 50$. In this scenario, performing numerical evaluation can take at least exponential time, and hence so is Householder's method.

\begin{algorithm}[t!]
	 \caption{Root-Finding algorithm}\label{RFA}
	\begin{algorithmic}[1]
	\Require $x_{0}\in\mathbb{R}$, $k\in \mathbb{N}$, $k\geq3$.
	\Ensure The function $A_{k}$ in Theorem \ref{maintheorem}.
		\Procedure{A}{$f(x)$, $x_{0}$, $k$}
		\State $\overrightarrow{\bf f} \gets \left(f|_{x=x_{0}},\ \cdots,\ f^{(k)}|_{x=x_{0}}\right)$
		\State $A_{0} \gets 1$
		\State $A_{1} \gets \overrightarrow{\bf f}[1]$
		\State $A_{2} \gets \overrightarrow{\bf f}[1]^{2}-0.5\overrightarrow{\bf f}[2]\overrightarrow{\bf f}[0]$
		\State $\overrightarrow{\bf v} \gets (A_{0}, A_{1}, A_{2})$
		\For{$2\leq s\leq k-1$}
		\State $\widehat{A}_{s}\gets  \sum_{l=2}^{s+1}\frac{1}{l!}(-\overrightarrow{\bf f}[0])^{l-2}\overrightarrow{\bf f}[l]A_{s+1-l}$
		\State $A_{s+1} \gets \overrightarrow{\bf f}[1]\overrightarrow{\bf v}[s]-\overrightarrow{\bf f}[0]\widehat{A}_{s}$
		\State $\overrightarrow{\bf v} \gets (A_{0},A_{1},A_{2},\cdots, A_{s+1})$
		\EndFor
		\State \Return $\overrightarrow{\bf v}[k]$
		\EndProcedure
	\end{algorithmic}
	\begin{algorithmic}[1]
		\Require Initial guess $x_{0}$, number of iterations $N$, order of convergence $k\geq 3$.
		\Ensure $N$-th element of sequence $x_{N}$.
		\State $x_{\text{old}} \gets x_{0}$
		\For{$0\leq j\leq N-1$}
		\State $x_{\text{new}}\gets x_{\text{old}}-\frac{f(x_{\text{old}})A(f,x_{\text{old}},k-1)}{A(f,x_{\text{old}},k)}$
		\State $x_{\text{old}} \gets x_{\text{new}}$
		\EndFor
	\end{algorithmic}
\end{algorithm}


There is a differential equation that relates $A_{k}$ with $\widehat{A}_{k}$. 
\medskip
\begin{proposition}\label{prop2.0.1}
  For $k\geq 2$, we have $A_{k}' = (k+1)\widehat{A}_{k}$.
\end{proposition}

\begin{proof}
  We will prove this result by induction on $k$. By the formula above, we see that 
  \begin{align*}
    A_{2} &= (f')^{2}-\frac{1}{2}f''f,\\
    \widehat{A}_{2} &= \frac{1}{2}f''f'-\frac{1}{6}ff^{(3)}.
  \end{align*}
  Thus we can easily verify that $A_{2}'=3\widehat{A}_{2}$, and the lemma holds for $k=2$. Assuming that the lemma holds for $A_{k}'=(k+1)\widehat{A}_{k}$, to show that $A_{k+1}'=(k+2)\widehat{A}_{k+1}$,
  \begin{align*}
    A_{k+1}' &= \big(f'A_{k}-f\widehat{A}_{k}\big)'=
    f''A_{k}+f'A_{k}'-f'\widehat{A}_{k}-f\widehat{A}_{k}'\\
    &= 
    f''A_{k}+(k+1)f'\widehat{A}_{k}-f'\widehat{A}_{k}-f\widehat{A}_{k}'\\
    &=
    f''A_{k}+kf'\widehat{A}_{k}-f\widehat{A}_{k}',
  \end{align*}
  where we use the induction hypothesis in the second line. It remains to compute $\widehat{A}_{k}'$. Using the equation for $\widehat{A}_{k}$, and differentiating it with respect to $x_{i}$, we see that 
  \begin{align*}
    \widehat{A}_{k}' &= \left(\sum_{l=2}^{k+1}\frac{1}{l!}(-f)^{l-2}f^{(l)}A_{k-l+1}\right)'\\
    &=
    \sum_{l=3}^{k+1}\frac{l-2}{l!}(-f)^{l-3}(-f')f^{(l)}A_{k-l+1} +\sum_{l=2}^{k+1}\frac{1}{l!}(-f)^{l-2}f^{(l+1)}A_{k-l+1}\\
&\hspace{0.5cm}+\sum_{l=2}^{k+1}\frac{1}{l!}(-f)^{l-2}f^{(l)}A_{k-l+1}'\\
    &=
    \sum_{l=3}^{k+1}\frac{l-2}{l!}(-f)^{l-3}(-f')f^{(l)}A_{k-l+1}+\sum_{l=2}^{k+1}\frac{1}{l!}(-f)^{l-2}f^{(l+1)}A_{k-l+1}\\
&\hspace{0.5cm}+\sum_{l=2}^{k+1}\frac{k-l+2}{l!}(-f)^{l-2}f^{(l)}\widehat{A}_{k-l+1},\\
  \end{align*}
  where again we use the induction hypothesis. Substituting this into the expression for $A_{k+1}'$, and also using the expression for $\widehat{A}_{k}$, we see that 
  \begin{align*}
    A_{k+1}' &= f''A_{k}+kf'\widehat{A}_{k}-f\widehat{A}_{k}'\\
    &=
    f''A_{k}+kf'\sum_{l=2}^{k+1}\frac{1}{l!}(-f)^{l-2}f^{(l)}A_{k-l+1}+\sum_{l=3}^{k+1}\frac{l-2}{l!}(-f)^{l-2}(-f')f^{(l)}A_{k-l+1}\\
    &\hspace{0.5cm}
    +\sum_{l=2}^{k+1}\frac{1}{l!}(-f)^{l-1}f^{(l+1)}A_{k-l+1}
    +\sum_{l=2}^{k+1}\frac{k-l+2}{l!}(-f)^{l-1}f^{(l)}\widehat{A}_{k-l+1}\\
    &=f''A_{k}+(\frac{1}{2}kf'f''-\frac{1}{2}ff''')A_{k-1}-\frac{1}{2}kff''\widehat{A}_{k-1}\\
    &\hspace{0.5cm}
    +\sum_{l=3}^{k+1}\frac{k-l+2}{l!}(-f)^{l-2}f^{(l)}f'A_{k-l+1}+\sum_{l=3}^{k+1}\frac{1}{l!}(-f)^{l-1}f^{(l+1)}A_{k-l+1}\\
    &\hspace{0.5cm}
    +\sum_{l=3}^{k+1}\frac{k-l+2}{l!}(-f)^{l-1}f^{(l)}\widehat{A}_{k-l+1}.
  \end{align*}
  Since
  \[
    \frac{k}{2}f'f''A_{k-1}-\frac{k}{2}ff''\widehat{A}_{k-1}=\frac{k}{2}f''A_{k},
  \]
  and
  \begin{align*}
    & \sum_{l=3}^{k+1}\frac{k-l+2}{l!}(-f)^{l-2}f^{(l)}f'A_{k-l+1}+\sum_{l=3}^{k+1}\frac{k-l+2}{l!}(-f)^{l-2}f^{(l)}(-f)\widehat{A}_{k-l+1}\\
    &\hspace{0.5cm}= \sum_{l=3}^{k+1}\frac{k-l+2}{l!}(-f)^{l-2}f^{(l)}A_{k-l+2} = \sum_{l=2}^{k}\frac{k-l+1}{(l+1)!}(-f)^{l-1}f^{(l+1)}A_{k-l+1},
  \end{align*}
  the expression of $A_{k+1}'$ may be rewritten as 
  \begin{align*}
    A_{k+1}' &= \frac{k+2}{2}f''A_{k}-\frac{1}{2}ff'''A_{k-1}+\sum_{l=2}^{k}\frac{k-l+1}{(l+1)!}(-f)^{l-1}f^{(l+1)}A_{k-l+1}+\sum_{l=3}^{k+1}\frac{1}{l!}(-f)^{l-1}f^{(l+1)}A_{k-l+1}\\
    &=
    \frac{k+2}{2}f''A_{k}+\frac{1}{2}(-f)f'''A_{k-1}+\frac{k-1}{3!}(-f)f'''A_{k-1}\\
    &\hspace{0.5cm}
    +\sum_{l=3}^{k}\left(\frac{k-l+1}{(l+1)!}(-f)^{l-1}f^{(l+1)}+\frac{1}{l!}(-f)^{l-1}f^{(l+1)}\right)A_{k-l+1}\\
&\hspace{0.5cm}
    +\frac{1}{(k+1)!}(-f)^{k}f^{(k+2)}A_{0}\\
    &=
    \frac{k+2}{2}f''A_{k}+\frac{k+2}{3!}(-f)f'''A_{k-1}+\sum_{l=3}^{k}\frac{k+2}{(l+1)!}(-f)^{l-1}f^{(l+1)}A_{k-l+1}\\
&\hspace{0.5cm}+\frac{k+2}{(k+2)!}(-f)^{k}f^{(k+2)}A_{0}\\
    &=
    (k+2)\left(\sum_{l=2}^{k+2}\frac{1}{l!}(-f)^{l-2}f^{(l)}A_{k+2-l}\right) = (k+2)\widehat{A}_{k+1},
  \end{align*} 
  and this finishes the proof.
\end{proof}

\subsection{Some discussions on the convergence order}
Following the procedure of the algorithm, we may compute the convergence factor $C_{k}$ in equation \eqref{rf-expr} as $x_{n}$ approaches $\alpha$. Then $|x_{n+1}-\alpha|\leq C_{k+1}|x_{n}-\alpha|^{k+1}$. The expressions of $C_{2}$ and $C_{3}$ come from Newton's and Halley's method, and whose expressions are well-known. For higher order ones, if we write $C_{k+1}=\lambda_{k+1}/\mu_{k+1}$, we have:
\begin{align*}
C_{4} &= \frac{18f'(f'')^{3}-24(f')^{2}f''f'''+4ff'(f''')^{2}+6(f')^{3}f^{(4)}-3ff'f''f^{(4)}}{24(6(f')^{4}-6f(f')^{2}f''+f^{2}f'f''')},\\
\lambda_{5} &= -180f^{(4)}+360f'f''f'''-80(f')^{2}(f''')^{2}-80ff''(f''')^{2}-120(f')^{2}f''f^{(4)}\\
&\hspace{0.5cm}+60f(f'')^{2}f^{(4)} +40ff'f'''f^{(4)}-5f^{2}(f^{(4)})^{2}+24(f')^{3}f^{(5)}-24ff'f''f^{(5)}\\
&\hspace{0.5cm}+4f^{2}f'''f^{(5)},\\
\mu_{5} &= 720(f')^{4}-1080f(f')^{2}f''+180f^{2}(f'')^{2}+240f^{2}f'f'''-30f^{3}f^{(4)},\\ 
\lambda_{6}&= 540(f'')^{5}-1440f'(f'')^{3}f'''+720(f')^{2}f''(f''')^{2}+360f(f'')^{2}(f''')^{2}\\
&\hspace{0.5cm}-160ff'(f''')^{3}+540(f')^{2}(f'')^{2}f^{(4)}-270f(f'')^{3}f^{(4)}-240(f')^{3}f'''f^{(4)}\\
&\hspace{0.5cm}-120ff'f''f'''f^{(4)}+20f^{2}(f''')^{2}f^{(4)}+30f(f')^{2}(f^{(4)})^{2}+30f^{2}f''(f^{(4)})^{2}\\
&\hspace{0.5cm}-144(f')^{3}f''f^{(5)}+144ff'(f'')^{2}f^{(5)}+48f(f')^{2}f'''f^{(5)}-48f^{2}f''f'''f^{(5)}\\
&\hspace{0.5cm}-12f^{2}f'f^{(4)}f^{(5)}+\frac{6}{5}f^{3}(f^{(5)})^{2}+24(f')^{4}f^{(6)}-36f(f')^{2}f''f^{(6)}\\
&\hspace{0.5cm}+6f^{2}(f'')^{2}f^{(6)}+8f^{2}f'f'''f^{(6)}-f^{3}f^{(4)}f^{(6)},\\
\mu_{6} &= 17280 (f')^{5}-34560f(f')^{3}f''+12960f^{2}f'(f'')^{2}+8640f^{2}f(f')^{2}f'''\\
&\hspace{0.5cm}-2880f^{3}f''f'''-1440f^{3}f'f^{(4)}+144f^{4}f^{(5)},
\end{align*}
and the list continues by induction.

\section{Application 1: Numerical Implementation of Householder's Method}
 
In this section we discuss how Theorem \ref{maintheorem} can be used to numerically implement Householder's method and give a more explicit description of its convergence factor. Recall that the Householder's method for solving $f(x)=0$ is given by an initial guess $x_{0}$ and a sequence of real numbers iteratively defined by:
\[
x_{n+1}=x_{n}+k\frac{(1/f)^{(k-1)}(x_{n})}{(1/f)^{(k)}(x_{n})},
\]
which guarantees convergence of $x_{n}$ to the nearest root $\alpha$ in the following manner
\[
|\alpha-x_{n+1}|\leq C|\alpha-x_{n}|^{k+1},
\]
for some constant $C$ which remains to be computed. The main result of this section is the following:
\medskip
\begin{theorem}\label{Theorem-Householder}
	Let $A_{k}$ be coefficients given by equation \eqref{mtAk}.  Then
	\[
	k\frac{(1/f)^{(k-1)}}{(1/f)^{(k)}} = -f\frac{A_{k-1}}{A_{k}}.
	\]
	In other words, the Householder's method can numerically evaluated by Algorithm \ref{RFA}.
\end{theorem}
Before we prove this theorem, we start with the following lemma.
\medskip
\begin{lemma}\label{lem5}
	Let $B_{k}$ be such that 
	\begin{eqnarray}\label{lem5stat}
	\left(\frac{1}{f}\right)^{(k)} = \frac{B_{k}}{f^{k+1}},
	\end{eqnarray}
	then $B_{k}=(-1)^{k}k!A_{k}$.
\end{lemma}

\begin{proof}
	We prove this by induction. The case for $k=2$ is just a direct verification. Suppose that the lemma is true for all $k$, then differentiating both sides of the expression of $(1/f)^{(k)}$ with respect to $x$, we obtain an expression of $B_{k+1}$ in terms of $B_{k}$:
	\begin{align*}
		B_{k+1} &= fB_{k}'-(k+1)f'B_{k}
		= (-1)^{k}(fk!A_{k}'-(k+1)!f'A_{k}),
	\end{align*}
	where we use the induction hypothesis. Using Proposition \ref{prop2.0.1}, we see that 
	\begin{align*}
		B_{k+1} = (-1)^{k}(k+1)!(f\widehat{A}_{k}-f'A_{k})
		= (-1)^{k+1}(k+1)!A_{k+1},
	\end{align*}
	and this finishes the proof.
\end{proof}

\begin{proof}[Proof of Theorem \ref{Theorem-Householder}]
	Using equation \eqref{lem5stat} and the result in Lemma \ref{lem5}:
	\begin{align*}
		k\frac{(1/f)^{(k-1)}(x_{n})}{(1/f)^{(k)}(x_{n})} = 
		kf\frac{B_{k-1}}{B_{k}} = 
		kf\frac{(-1)^{k-1}(k-1)!A_{k-1}}{(-1)^{k}k!A_{k}}
		=
		-f\frac{A_{k-1}}{A_{k}},
	\end{align*}
	which concludes the proof.
\end{proof} 
\medskip
\begin{corollary}
	Let $f(x)=0$ be an equation and $x_{0}=x_{0}'=x_{0}''$ be the first guess of the actual solution $\alpha$. Let $x_{n}'$ and $x_{n}''$ be sequences given  by
	\begin{equation}
		\begin{aligned}
			x_{n+1}' = x_{n}' + k\frac{(1/f)^{(k-1)}(x_{n}')}{(1/f)^{(k)}(x_{n}')},\qquad
			x_{n+1}'' = x_{n}'' -\frac{f(x_{n}'')A_{k-1}(x_{n}'')}{A_{k}(x_{n}'')}.
		\end{aligned}
	\end{equation}
	Let $C$ be such that 
	\[
	|x_{n+1}''-\alpha|\leq C(x_{n}'',\alpha)|x_{n}''-\alpha|^{k+1}.
	\]
	Then for the same $C$,
	\[
	|x_{n+1}'-\alpha|\leq C(x_{n}'',\alpha)|x_{n}'-\alpha|^{k+1}.
	\]
\end{corollary}
\begin{proof}
	By Theorem \ref{Theorem-Householder}, $x_{n}''=x_{n}'$ for all $n$. 
\end{proof}

\section{Application 2: Numerical Lower Bound of Minimal Distance of a Random Linear Code}
\subsection{Preliminaries and Motivation for Solving $H_{q}(x)=y$}
In coding theory, finding the minimal distance of a linear code is crucial in determining if one will be able to decode a message. However, it is an NP-complete problem. We will see that solving the equation $H_{q}(x)=y$, where $q$ is a prime number and $H_{q}$ is the  $q$-ary entropy function, plays an important role in computing a lower bound of the minimal distance of a random linear code. We first  recall some preliminary notions \cite{Meurer-2012, Barg-1997, vL98}. 

Let $\mathbb{F}_{q}^{n}$ be an $n$-dimensional space over finite field with $q$ elements. A $q$-ary $[n,k]$ linear code $\mathcal{C}\subseteq \mathbb{F}_{q}^{n}$ is a vector subspace of dimension $k$. The ratio $R=\lim_{n\rightarrow\infty}\frac{k}{n}$ is called the rate  of $\mathcal{C}$. The space $\mathbb{F}_{q}^{n}$ is equipped with the Hamming weight
\[
  \text{wt}_{H}(x)=\#\{x_{i}:\ x_{i}\neq 0\},\qquad x=(x_{1},\cdots,x_{n})\in\mathbb{F}_{q}^{n},
\]
which defines the Hamming distance between any two vectors:
\[
  d_{H}(x,y):=\text{wt}_{H}(x-y).
\]
\begin{definition}
  The minimal distance $d$ of an $[n,k]$-linear code $\mathcal{C}\subseteq \mathbb{F}_{q}^{n}$ is the minimum distance of every pair of elements in $\mathcal{C}$:
  \[
    d := d(\mathcal{C})=\min_{x,y\in\mathcal{C}} d_{H}(x,y).
  \]
\end{definition}

Given $w\in \mathbb{N}$, we let $B_{q}(n,w)$ be the set of all elements in $\mathbb{F}_{q}^{n}$ of Hamming weight less than or equal to $w$. It is a ball that is centered at ${\bf 0}\in\mathbb{F}_{q}^{n}$.

Let $G$ be a generating matrix of the linear code $\mathcal{C}$ of dimension $k$. It is a matrix consisting of rows that form a basis of the linear code $\mathcal{C}$. It gives an injective encoding function 
\begin{equation}
\begin{aligned}
\text{Encode}:\ \mathbb{F}_{q}^{k} &\longrightarrow \mathbb{F}_{q}^{n}\\
m &\longmapsto mG.
\end{aligned}
\end{equation}
Given an element $x\in\mathbb{F}_{q}^{n}$, a decoding function $\text{Decode}:\ \mathbb{F}_{q}^{n}\longrightarrow\mathcal{C}$ should be a function that returns a codeword $c\in\mathcal{C}$ that is closest to $x$ amongst all the choices in $\mathcal{C}$. In other words, $d(x,\text{Decode}(x))=d(x,\mathcal{C})$. 

A sender perturbs the codeword with a small vector $e\in\mathbb{F}_{q}^{n}$ and sends the message $x=mG+e$, where the error $e$ is chosen from the ball $B_{q}(n,w)$ of some fixed radius $w$; while the task of its recipient is to recover $m$. The parity check matrix $H$ is an $(n-k)\times n$ matrix satisfying $GH^{T}=0$. Thus the message is a codeword if and only if $xH^{T}=0$, and the presence of $e$ may cause $xH^{T}$ to be non-zero. In this case, one is required to solve the syndrome decoding problem $s:=xH^{T}=eH^{T}$ for $e$, and the recipient will then be able to recover $m$.

When the sender perturbs the codeword, he needs to determine a good set $B_{q}(n,w)$ of  error  vectors so that he would be able to perform encryption  $mG+e$ without ambiguity. It is possible for different vectors $m_{1}$, $m_{2}$, and for different errors  $e_{1}$, $e_{2}$, to give the same sum $v:=m_{1}G+e_{1}=m_{2}G+e_{2}$. This happens when the set intersection of $m_{1}G+B_{q}(n,w)$ and $m_{2}G+B_{q}(n,w)$ is non-empty, and contains $v$. The recipient will therefore be unable to determine if the original message is $m_{1}$ or $m_{2}$.

Thus to avoid this difficulty, it is necessary to define an error set based on the minimal distance $d_{H}(\mathcal{C})$ of a linear code $\mathcal{C}$. Indeed, we observe that if $d_{H}(\mathcal{C})>0$, then for $0<\lambda\leqslant d_{H}(\mathcal{C})-1$, any two sets $m_{1}G+B_{q}(n,\frac{\lambda}{2})$,  $m_{2}G+B_{q}(n,\frac{\lambda}{2})$ will be disjoint. However, sometimes in practice, one deals with linear codes of very large dimensions, and finding the exact minimum distance may not be computationally feasible. From this point of view, one asks if a good choice of the radius $w$ for the ball $B_{q}(n,w)$ can be found for \textit{any randomly chosen} $[n,k]$-linear codes. The following classical result answers the question.

\medskip
\begin{theorem}[Gilbert-Varshamov theorem]\label{GV}
	For $0<R<1$, let $D_{\text{GV}}(R)$ be the unique solution $x\in [0,\frac{q-1}{q}]$ that satisfies
	\begin{equation}\label{GVeq}
	H_{q}(x)=1-R,
	\end{equation}
	where $H_{q}$ is the $q$-ary entropy function in Definition \ref{Hq}. Then for almost all linear codes $\mathcal{C}$ of rate $R=\lim_{n\rightarrow \infty}\frac{k}{n}$, it holds
	\[
	d_{H}(\mathcal{C})\geq \lfloor D_{\text{GV}}(R)n\rfloor.\qed
	\]
\end{theorem}
Thanks to this theorem, one will have a high chance of decoding success if the error $e$ lies in $B_{q}(n,0.5\cdot\lfloor  D_{GV}(R)n\rfloor)$. In general, equation \eqref{GVeq} has no closed-form solution except for the trivial cases $R=0$ or $R=1$, and will thus  require numerical methods. We will therefore solve $H_{q}(x)=y$ for various values of $y$ using  Theorem \ref{maintheorem}, and compare its performance with Householder's method. When the order of convergence is high, Theorem \ref{maintheorem} can return results much faster than Householder's method, which sometimes fails to compute. 

We will first focus on the binary entropy function $H_{2}(x)$, which appears quite commonly when working with the finite field $\mathbb{F}_{2}$, before looking at the general case $H_{q}$. For $0\leq y \leq 1$, let $\alpha$ be the solution in $\left[0,\frac{q-1}{q}\right]$ such that $H_{q}(\alpha)=y$. For any $x\in \left[0,\frac{q-1}{q}\right]$, mean value theorem gives
\[
x-\alpha = \frac{H_{q}(x)-y}{H_{q}'(\xi)}, 
\]
where $\xi$ lies between $x$ and $\alpha$. When $x$ and $\alpha$ are very close to each other, we therefore have an approximation 
\begin{equation}\label{x-alpha}
|x-\alpha|\approx \frac{\left|H_{q}(x)-y\right|}{\left|H_{q}'(x)\right|},
\end{equation}
which then allows us to estimate the closeness of $x_{n}$ the $n$-th iterated value of the root-finding method to the actual solution $\alpha$ despite not knowing its value. The numerical computations are performed on SageMath with 10000 bits of precision, or about 3010 decimal places. 


\subsection{Root-Finding methods on $H_{2}(x)=y$}

In this subsection we will compute numerical solutions to $H_{2}(x)=y$, where $0.1\leq y\leq 0.9$ at $0.1$-interval, and $0\leq x\leq 0.5$. Equation \ref{H2-table1} shows their respective numerical values. 

\begin{equation}\label{H2-table1}
	\begin{aligned}
	H_{2}(0.01298686205551...)&=0.1,\qquad 
	H_{2}(0.03112446030478...) =0.2,\\
	H_{2}(0.05323904077679...)&=0.3,\qquad
	H_{2}(0.07938260048064...)=0.4,\\
	H_{2}(0.11002786443836...)&=0.5,\qquad
	H_{2}( 0.14610240341189...)=0.6,\\
	H_{2}(0.18929770537063...)&=0.7,\qquad
	H_{2}(0.24300385380895...)=0.8,\\
	H_{2}(0.31601934632361...)&=0.9.\\
	\end{aligned}
\end{equation}

Since a general $q$-ary entropy function can be expressed in terms of the binary entropy function $H_{2}$, we will focus on $H_{2}$ first.  By a theorem of Tops{\o}e \cite{flemming_bounds_2001}, the binary entropy function is bounded above and below by 
\[
4x(1-x)\leq H_{2}(x) \leq (4x(1-x))^{\frac{1}{\ln 4}}
\]
on the interval $0\leq x\leq 0.5$. To simplify notations, we write $f(x)=4x(1-x)$ and $g(x)=(4x(1-x))^{\frac{1}{\ln 4}}$, and let $f^{-1}$, $g^{-1}$ be their respective inverses. We let $H_{2}^{-1}$ be the inverse of $H_{2}$. Therefore, we have
\[
g^{-1}(x)\leq H_{2}^{-1}(x)\leq f^{-1}(x), \quad 0\leq x\leq 1. 
\]

\begin{proposition}\label{propjus}
	Let $y\in [0,1]$ be fixed, and let $\alpha$, $\beta$, $\gamma$ be such that $f(\alpha)=H_{2}(\beta)=g(\gamma)=y$. Write $\sigma := \ln(4)^{\frac{1}{1-\ln(4)}}$. Then 
	\[
	|\alpha-\gamma|\leq \frac{\sigma-\sigma^{\ln(4)}}{2(\sqrt{1-0.5^{\ln(4)}}+\sqrt{0.5})}\approx 0.04007.
	\]
	Hence the distance between any two of $\alpha$, $\beta$, $\gamma$ is bounded above by $0.04007$.
	\end{proposition}

\begin{proof}
	We have explicitly 
	\begin{equation}
		\begin{aligned}
			f^{-1}(x) &= \frac{1}{2}\left(1-\sqrt{1-y}\right),\qquad
			g^{-1}(x) = \frac{1}{2}\left(1-\sqrt{1-y^{\ln(4)}}\right).
		\end{aligned}
	\end{equation}
	Therefore 
	\[
	f^{-1}-g^{-1} = \frac{y-y^{\ln(4)}}{2\left(\sqrt{1-y^{\ln(4)}}+\sqrt{1-y}\right)}.
	\]
	The numerator achieves maximum when $y=\sigma$, while the denominator is a decreasing function and hence is bounded below by $\sqrt{1-0.5^{\ln(4)}}+\sqrt{0.5}$. Combining the two inequalities, we conclude the proof. 
\end{proof}

The actual upper bound is in fact close to $0.039989$. Proposition \ref{propjus} provides a theoretical justification for using either $\alpha$ or $\gamma$ as a starting point for root-finding algorithm. A more careful observation shows that the function $(4x(1-x))^{\frac{1}{\ln(4)}}$ is a better fit to $H_{2}(x)$. Indeed, numerical results show that $|\alpha-\beta|\leq 0.039989$ while $|\gamma-\beta|\leq 0.00379$. Hence the convergence rate with $(4x(1-x))^{\frac{1}{\ln(4)}}$ is expected to be much faster.  Table \ref{H2table1} shows that it is indeed the case.  Moreover in both cases, Algorithm \ref{RFA} runs faster than Householder's method. It is observed that for order $k=100$, Algorithm \ref{RFA} takes less than 2 minutes to complete, while Householder's method runs for a very long time. Their running times are illustrated in Figure \ref{runningtimes} as given by the values in Table \ref{H2table1}. Based on curve-fitting, it shows that the running times of our proposed method is approximately $-2.2\times 10^{-12}\ k^{3}+0.02k^2-0.10k+3$ for $f=4x(1-x)$, and $0.003k^3+0.1k^2-0.83k+6$ for $g=(4x(1-x))^{\frac{1}{\ln(4)}}$. On the other hand the respective running times of complexities of Householder's method grow according to  $2.4\exp(0.267k)$ for $f$, and $2.49\exp(0.259k)$ for $g$. This illustrates the point made earlier that the time complexity of our proposed method is $O(k^{3})$ while the time complexity of Householder's method is expected to run at least in $\exp(0.24k)$ time.

\begin{table}[t!]
	\centering 
	

	\begin{tabular}{|c|cccccccc|}
		\hline
		\multirow{3}{*}{$y$} & \multicolumn{8}{c|}{$|H_{2}(x_{1})-y|$ with $H_{2}$ approximated by $f$ or $g$}                                                                                                                                                                                                                                 \\ \cline{2-9} 
		& \multicolumn{2}{c|}{$5$}                                             & \multicolumn{2}{c|}{$10$}                                             & \multicolumn{2}{c|}{$15$}                                             & \multicolumn{2}{c|}{$20$}                        \\ \cline{2-9} 
		& \multicolumn{1}{c|}{$f$}         & \multicolumn{1}{c|}{$g$}          & \multicolumn{1}{c|}{$f$}          & \multicolumn{1}{c|}{$g$}          & \multicolumn{1}{c|}{$f$}          & \multicolumn{1}{c|}{$g$}          & \multicolumn{1}{c|}{$f$}          & $g$          \\ \hline
		$0.1$                & \multicolumn{1}{c|}{\small{1.7 e-5}}  & \multicolumn{1}{c|}{\small{6.4 e-8}}  & \multicolumn{1}{c|}{\small{1.5 e-7}}  & \multicolumn{1}{c|}{\small{1.7 e-11}} & \multicolumn{1}{c|}{\small{2.1 e-9}}  & \multicolumn{1}{c|}{\small{7.9 e-15}} & \multicolumn{1}{c|}{\small{3.5 e-11}} & \small{4.5 e-18} \\ \hline
		$0.2$                & \multicolumn{1}{c|}{\small{9.2 e-6}} & \multicolumn{1}{c|}{\small{2.7 e-9}}  & \multicolumn{1}{c|}{\small{3.1 e-8}}  & \multicolumn{1}{c|}{\small{2.5 e-14}} & \multicolumn{1}{c|}{\small{1.7 e-10}} & \multicolumn{1}{c|}{\small{3.8\ e-19}} & \multicolumn{1}{c|}{\small{1.2 e-12}} & \small{7.2 e-24} \\ \hline
		$0.3$                & \multicolumn{1}{c|}{\small{4.4 e-6}} & \multicolumn{1}{c|}{\small{1.8 e-10}} & \multicolumn{1}{c|}{\small{6.4 e-9}}  & \multicolumn{1}{c|}{\small{1.1 e-16}} & \multicolumn{1}{c|}{\small{1.6 e-11}} & \multicolumn{1}{c|}{\small{1.1 e-22}} & \multicolumn{1}{c|}{\small{4.7 e-14}} & \small{1.5 e-28} \\ \hline
		$0.4$                & \multicolumn{1}{c|}{\small{1.9 e-6}} & \multicolumn{1}{c|}{\small{1.2 e-11}} & \multicolumn{1}{c|}{\small{1.2 e-9}}  & \multicolumn{1}{c|}{\small{6.3 e-19}} & \multicolumn{1}{c|}{\small{1.3 e-12}} & \multicolumn{1}{c|}{\small{5.5 e-26}} & \multicolumn{1}{c|}{\small{1.8 e-15}} & \small{5.9 e-33} \\ \hline
		$0.5$                & \multicolumn{1}{c|}{\small{7.5 e-7}} & \multicolumn{1}{c|}{\small{7.5 e-13}} & \multicolumn{1}{c|}{\small{2.0 e-10}} & \multicolumn{1}{c|}{\small{3.1 e-21}} & \multicolumn{1}{c|}{\small{9.0 e-14}} & \multicolumn{1}{c|}{\small{2.2\ e-29}} & \multicolumn{1}{c|}{\small{5.0 e-17}} & \small{1.9 e-37} \\ \hline
		$0.6$                & \multicolumn{1}{c|}{\small{2.9 e-7}} & \multicolumn{1}{c|}{\small{3.5 e-14}} & \multicolumn{1}{c|}{\small{2.4 e-11}} & \multicolumn{1}{c|}{\small{9.3 e-24}} & \multicolumn{1}{c|}{\small{4.2 e-15}} & \multicolumn{1}{c|}{\small{4.4 e-33}} & \multicolumn{1}{c|}{\small{9.0 e-19}} & \small{2.6 e-42} \\ \hline
		$0.7$                & \multicolumn{1}{c|}{\small{1.4 e-7}} & \multicolumn{1}{c|}{\small{1.5 e-15}} & \multicolumn{1}{c|}{\small{1.8 e-12}} & \multicolumn{1}{c|}{\small{9.9 e-27}} & \multicolumn{1}{c|}{\small{1.0 e-16}} & \multicolumn{1}{c|}{\small{2.0 e-37}} & \multicolumn{1}{c|}{\small{7.0 e-21}} & \small{5.2 e-48} \\ \hline
		$0.8$                & \multicolumn{1}{c|}{\small{1.0 e-7}} & \multicolumn{1}{c|}{\small{3.2 e-17}} & \multicolumn{1}{c|}{\small{9.2 e-14}} & \multicolumn{1}{c|}{\small{8.9 e-31}} & \multicolumn{1}{c|}{\small{1.0 e-18}} & \multicolumn{1}{c|}{\small{5.0 e-43}} & \multicolumn{1}{c|}{\small{1.1 e-23}} & \small{2.2 e-55} \\ \hline
		$0.9$                & \multicolumn{1}{c|}{\small{7.7 e-8}} & \multicolumn{1}{c|}{\small{2.5 e-19}} & \multicolumn{1}{c|}{\small{1.9 e-13}} & \multicolumn{1}{c|}{\small{1.6 e-34}} & \multicolumn{1}{c|}{\small{4.7 e-19}} & \multicolumn{1}{c|}{\small{1.1 e-49}} & \multicolumn{1}{c|}{\small{1.2 e-24}} & \small{7.1 e-65} \\ \hline
	\end{tabular}
\medskip

\begin{tabular}{|c|cccccccc|}
	\hline
	\multirow{3}{*}{$y$} & \multicolumn{8}{c|}{$|x_{1}-\alpha|$ for $H_{2}(\alpha)=y$}                                                                                                                                                                             \\ \cline{2-9} 
	& \multicolumn{2}{c|}{$5$}                                     & \multicolumn{2}{c|}{$10$}                                     & \multicolumn{2}{c|}{$15$}                                     & \multicolumn{2}{c|}{$20$}                \\ \cline{2-9} 
	& \multicolumn{1}{c|}{$f$}     & \multicolumn{1}{c|}{$g$}      & \multicolumn{1}{c|}{$f$}      & \multicolumn{1}{c|}{$g$}      & \multicolumn{1}{c|}{$f$}      & \multicolumn{1}{c|}{$g$}      & \multicolumn{1}{c|}{$f$}      & $g$      \\ \hline
	$0.1$                & \multicolumn{1}{c|}{2.7 e-6} & \multicolumn{1}{c|}{1.0 e-8}  & \multicolumn{1}{c|}{2.3 e-8}  & \multicolumn{1}{c|}{2.7 e-12} & \multicolumn{1}{c|}{3.4 e-10} & \multicolumn{1}{c|}{1.3 e-15} & \multicolumn{1}{c|}{5.7 e-12} & 7.2 e-19 \\ \hline
	$0.2$                & \multicolumn{1}{c|}{1.8 e-6} & \multicolumn{1}{c|}{5.5 e-10} & \multicolumn{1}{c|}{6.2 e-9}  & \multicolumn{1}{c|}{5.0 e-15} & \multicolumn{1}{c|}{3.4 e-11} & \multicolumn{1}{c|}{7.6 e-20} & \multicolumn{1}{c|}{2.4 e-13} & 1.5 e-24 \\ \hline
	$0.3$                & \multicolumn{1}{c|}{1.1 e-6} & \multicolumn{1}{c|}{4.2 e-11} & \multicolumn{1}{c|}{1.6 e-9}  & \multicolumn{1}{c|}{2.6 e-17} & \multicolumn{1}{c|}{3.8 e-12} & \multicolumn{1}{c|}{2.8 e-23} & \multicolumn{1}{c|}{1.1 e-14} & 3.6 e-29 \\ \hline
	$0.4$                & \multicolumn{1}{c|}{5.4 e-7} & \multicolumn{1}{c|}{3.5 e-12} & \multicolumn{1}{c|}{3.5 e-10} & \multicolumn{1}{c|}{1.8 e-19} & \multicolumn{1}{c|}{3.8 e-13} & \multicolumn{1}{c|}{1.6 e-26} & \multicolumn{1}{c|}{5.0 e-16} & 1.7 e-33 \\ \hline
	$0.5$                & \multicolumn{1}{c|}{2.5 e-7} & \multicolumn{1}{c|}{2.5 e-13} & \multicolumn{1}{c|}{6.6 e-11} & \multicolumn{1}{c|}{1.0 e-21} & \multicolumn{1}{c|}{3.0 e-14} & \multicolumn{1}{c|}{7.4 e-30} & \multicolumn{1}{c|}{1.7 e-17} & 6.4 e-38 \\ \hline
	$0.6$             & \multicolumn{1}{c|}{1.1 e-7} & \multicolumn{1}{c|}{1.4 e-14} & \multicolumn{1}{c|}{9.5 e-12} & \multicolumn{1}{c|}{3.7 e-24} & \multicolumn{1}{c|}{1.7 e-15} & \multicolumn{1}{c|}{1.7 e-33} & \multicolumn{1}{c|}{3.5 e-19} & 1.0 e-42 \\ \hline
	$0.7$                & \multicolumn{1}{c|}{6.6 e-8} & \multicolumn{1}{c|}{5.5 e-16} & \multicolumn{1}{c|}{8.5 e-13} & \multicolumn{1}{c|}{4.7 e-27} & \multicolumn{1}{c|}{5.0 e-17} & \multicolumn{1}{c|}{9.7 e-38} & \multicolumn{1}{c|}{3.2 e-21} & 2.5 e-48 \\ \hline
	$0.8$                & \multicolumn{1}{c|}{6.3 e-8} & \multicolumn{1}{c|}{2.0 e-17} & \multicolumn{1}{c|}{5.6 e-14} & \multicolumn{1}{c|}{5.4 e-31} & \multicolumn{1}{c|}{6.1 e-19} & \multicolumn{1}{c|}{3.1 e-43} & \multicolumn{1}{c|}{6.8 e-24} & 1.4 e-55 \\ \hline
	$0.9$                & \multicolumn{1}{c|}{6.9 e-8} & \multicolumn{1}{c|}{2.2 e-19} & \multicolumn{1}{c|}{1.7 e-13} & \multicolumn{1}{c|}{1.5 e-34} & \multicolumn{1}{c|}{4.2 e-19} & \multicolumn{1}{c|}{9.8 e-50} & \multicolumn{1}{c|}{1.0 e-24} & 6.4 e-65 \\ \hline
	Thm \ref{maintheorem}              & \multicolumn{1}{c|}{3 s}     & \multicolumn{1}{c|}{4 s}      & \multicolumn{1}{c|}{4 s}      & \multicolumn{1}{c|}{5 s}      & \multicolumn{1}{c|}{6 s}      & \multicolumn{1}{c|}{7 s}      & \multicolumn{1}{c|}{9 s}      & 8 s      \\ \hline
	Householder          & \multicolumn{1}{c|}{9 s}     & \multicolumn{1}{c|}{8 s}      & \multicolumn{1}{c|}{34 s}     & \multicolumn{1}{c|}{32 s}     & \multicolumn{1}{c|}{132 s}    & \multicolumn{1}{c|}{122 s}    & \multicolumn{1}{c|}{500 s}    & 442 s    \\ \hline
\end{tabular}

\medskip

\caption{Tables showing the performance of Algorithm \ref{RFA}  applied to the equation $H_{2}(x_{1})=y$ where $y$ runs from $0.1$ to $0.9$, and for various orders $5$, $10$, $15$, $20$. The initial guesses $x_{0}$ are obtained by solving $f=y$ or $g=y$, where $f=4x(1-x)$ and $g=(4x(1-x))^{\frac{1}{\ln 4}}$ are approximating functions of $H_{2}(x)$. The accuracy of $x_{1}$ after 1 iteration are studied based on $|H_{2}(x_{1})-y|$ and $|x_{1}-\alpha|$. Their running times of Algorithm \ref{RFA} and the Householder's method are shown in seconds.  }\label{H2table1}
\end{table}

\begin{figure}[h!]
	
	\centering
	\includegraphics[scale=0.42]{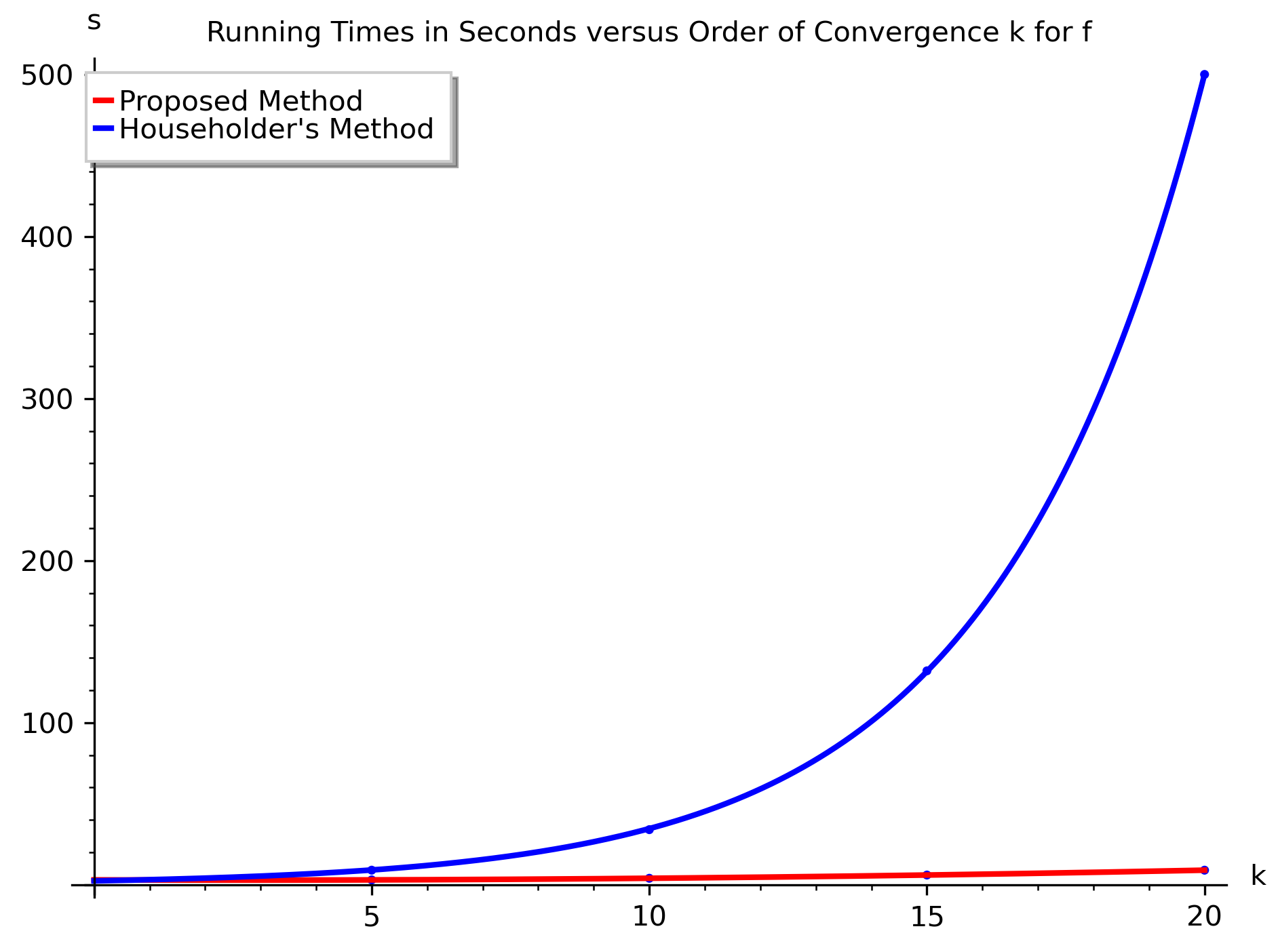}
	\includegraphics[scale=0.42]{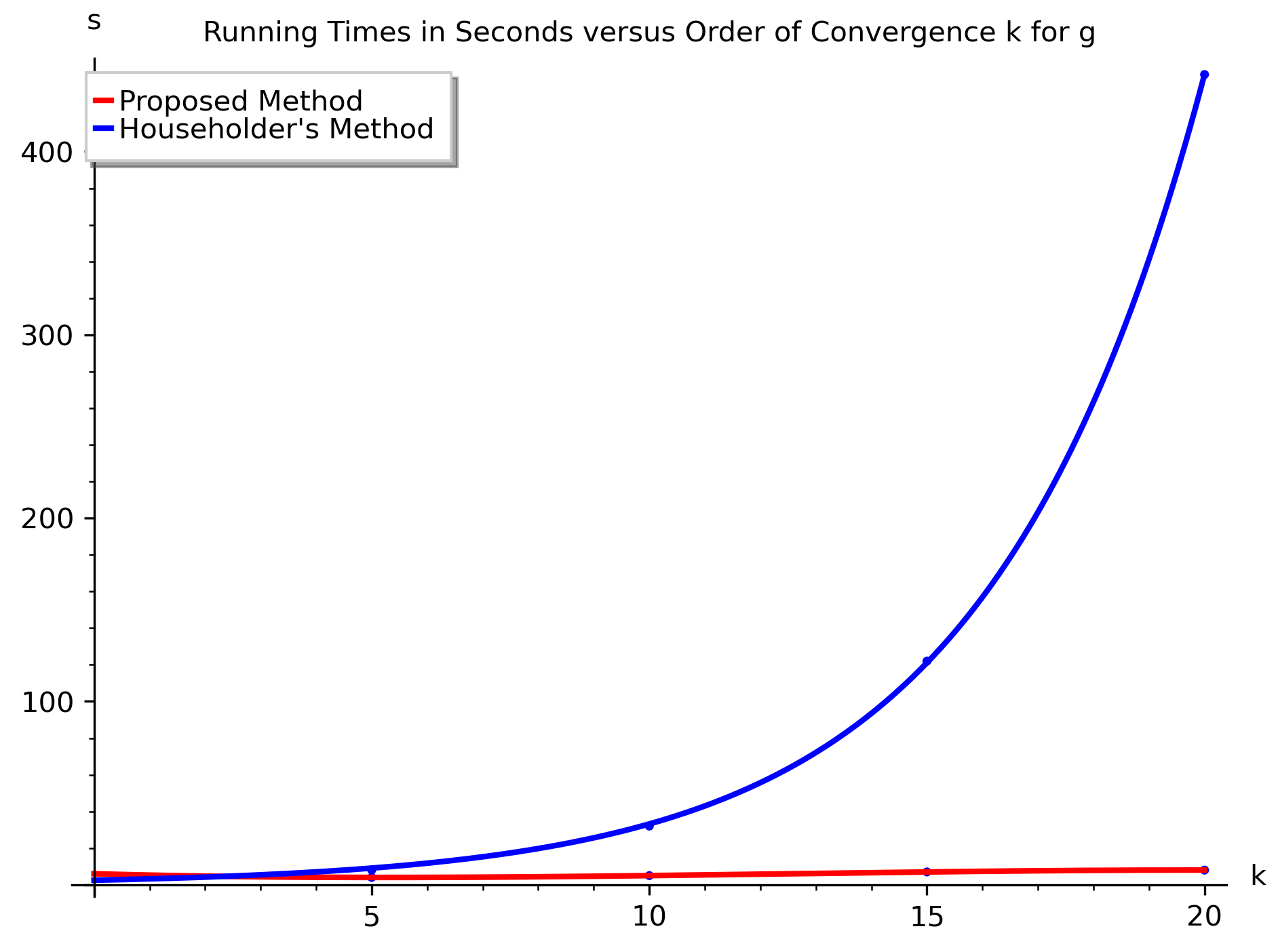}
	\caption{Running times based on Table \ref{H2table1} for $f=4x(1-x)$ and $g=(4x(1-x))^{\frac{1}{\ln(4)}}$.}\label{runningtimes}
\end{figure}

\subsection{Root-Finding methods on $H_{q}(x)=y$ for $q\geq 3$}

Most of the discussion above for the binary entropy function case can be applied to the general case of solving preimage of a $q$-ary entropy function $H_{q}(x)=y$. Indeed, one observes that $H_{q}$ may be written in terms of $H_{2}$:

\[H_{q}(x) = x\log_{q}(q-1)+\frac{H_{2}(x)}{\log_{2}q},\]
and thus any function $\varphi$ that approximates $H_{2}$ may be used to approximate $H_{q}$:
\[
\widetilde{\varphi}:=x\log_{q}(q-1)+\frac{\varphi(x)}{\log_{2}q}.
\]
The following proposition shows that if $\varphi$ approximates $H_{2}$ in $\sup$-norm, then the solutions to $\widetilde{\varphi}=y$ and $H_{q}=y$ will be close to each other. In fact, their solutions will even be closer to each other if the prime number $q$ is large. 
\medskip
\begin{proposition}\label{prop401}
	Let $\varphi$ be a function on the interval $[0.0.5]$ such that $\varphi(0)=0$, $\varphi(0.5)=1$, and 
	\[
	\sup_{x\in [a,b]}|\varphi-H_{2}|\leq \varepsilon
	\]
	for some small $\varepsilon$. Let $\alpha$ and $\beta$ be such that 
	\begin{equation}\label{propHq}
		\alpha\log_{q}(q-1)+\frac{\varphi(\alpha)}{\log_{2}q}=y=H_{q}(\beta).
	\end{equation}
	Then 
	\begin{equation}\label{propbound}
	|\alpha-\beta|\leq \frac{\varepsilon}{\log_{2}(q-1)}.
	\end{equation}
\end{proposition}
\begin{proof}
	Bringing the right hand side of the equation \eqref{propHq} over, and using the expression of $H_{q}$, we see that 
	\begin{align}
		\left(\alpha-\beta\right)\log_{q}(q-1)-\frac{\varphi(\alpha)-H_{2}(\beta)}{\log_{2}q}=0.
	\end{align}
	Therefore, we have 
	\begin{align}
		|\alpha-\beta| = \frac{|\varphi(\alpha)-H_{2}(\beta)|}{\log_{q}(q-1)\log_{2}q} =
		\frac{|\varphi(\alpha)-H_{2}(\beta)|}{\log_{2}(q-1)} \leq 
		\frac{\varepsilon}{\log_{2}(q-1)}.
	\end{align}
	This finishes the proof. 
\end{proof}

Recall that we have the inequality
$
f(x) \leq H_{2}(x) \leq g(x),
$
where $f(x) = 4x(1-x)$ and $g(x) = (4x(1-x))^{\frac{1}{\ln(4)}}$. We will thus be able to estimate $\varepsilon$ in the previous proposition. 

\begin{proposition}\label{prope}
	Let 
	\[\eta := \frac{1}{2}\left(1-\left(1-\ln(4)^{\frac{\ln(4)}{1-\ln(4)}}\right)^{1/2}\right),\] 
	and $\varepsilon := |f(\eta)-g(\eta)|\approx 0.1196$. 
	Let $\alpha$, $\beta$ and $\gamma$ be such that $f(\alpha)=H_{q}(\beta)=g(\gamma)$. Then 
	\begin{equation}
		\begin{aligned}
			|\alpha-\beta| &\leq \frac{\varepsilon}{\log_{2}(q-1)},\qquad
			|\gamma - \beta| & \leq \frac{\varepsilon}{\log_{2}(q-1)}.
		\end{aligned}
	\end{equation}
\end{proposition}

\begin{proof}
	It suffices to find the maximum of the function 
	\[
	g-f = (4x(1-x))^{\frac{1}{\ln(4)}}-4x(1-x).
	\]
	Differentiating and solving $g'-f'=0$ for $x$, we see that the maximum occurs at $x=\eta$. The previous proposition concludes the proof. 
\end{proof}

Proposition \ref{prope} shows that the first guess  of the root can be obtained by either solving $\tilde{f}(x)=y$ or $\tilde{g}(x)=y$ for $x$. In addition, the bound \eqref{propbound} goes to zero whenever $q$ is large, which implies that the first guess becomes closer to the actual solution. However, $\tilde{f}(x)=y$ is a quadratic equation in $x$, which may be preferred to $\tilde{g}$ which may not have closed-form solutions.

\subsubsection{For large primes $q$}
When $q$ is large, the function $H_{q}(x)$ approximates to $x$. This allows us to find an alternative candidate to approximate the root of $H_{q}(x)=y$ without using non-linear functions. 
\medskip
\begin{proposition}\label{prop402}
Let $\varepsilon>0$ be given. If $q\geq 2^{\frac{2}{\varepsilon}}$, then for all $x\in [0,\frac{q-1}{q}]$,
\[
\left|H_{q}(x)-x\right|\leq \varepsilon.
\]
\end{proposition}

\begin{proof}
In fact, since $|H_{2}(x)|\leq 1$ for all $x$, given $q\geq 2^{\frac{2}{\varepsilon}}$, we have 
\[
\frac{H_{2}(x)}{\log_{2}q}\leq \frac{1}{\log_{2} q}\leq \frac{1}{2}\varepsilon.
\]
In addition, 
\[
\log_{q}(q-1)=\frac{\log_{2}(q-1)}{\log_{2} q}=1+\left(\frac{1}{\log_{2} q}\log_{2}\frac{q-1}{q}\right).
\]
Since $q\geq 2$, $|\log_{2}\frac{q-1}{q}|\leq 1$. Thus 
\begin{align*}
x-H_{q}(x) &= x(1-\log_{q}(q-1)) - \frac{H_{2}(x)}{\log_{2}(q)}
= -\frac{x}{\log_{2} q}\log_{2}\frac{q-1}{q}-\frac{H_{2}(x)}{\log_{2} q}.
\end{align*}
As $x$ is bounded above by $1$, it follows that 
$
|x-H_{q}(x)|\leq \frac{1}{2}\varepsilon + \frac{1}{2}\varepsilon = \varepsilon,
$
which finishes the proof. 
\end{proof}

The magnitude of prime $p$ to achieve an approximation with error less than $\varepsilon$ can be very huge. For instance, in order to obtain an error of less than $\varepsilon=0.001$, we would require $p>2^{2000}$.

Given such large primes, to solve for $H_{q}(x)=y$, we may let $x_{0}=y$ be the initial value, and apply the algorithm in Section \ref{section2} to obtain the solution. To provide a numerical illustration, we use one of the Mersenne primes $2^{127}-1$. According to Proposition \ref{prop402}, approximately we have $|H_{q}(x)-x|\leq 0.0157$ for all $x$. We perform 1 iteration for various values of $k$, with the  values $y$ running through from $0.1$ to $0.9$ at $0.1$-interval. Table \ref{Hq-x} shows the accuracy values and a comparison of running times  where we see that in this case Algorithm \ref{RFA} is significantly faster than Householder's method.

\begin{table}[t]
	\centering
\begin{tabular}{|c|cccc|}
\hline
\multirow{2}{*}{$y$} & \multicolumn{4}{c|}{$|H_{q}(x_{1})-y|$, $q=2^{127}-1$}                                                                                                                      \\ \cline{2-5} 
                     & \multicolumn{1}{c|}{$5$}                   & \multicolumn{1}{c|}{$10$}                  & \multicolumn{1}{c|}{$15$}                  & $20$                  \\ \hline
$0.1$                & \multicolumn{1}{c|}{$2.01\times 10^{-14}$} & \multicolumn{1}{c|}{$1.22\times 10^{-22}$} & \multicolumn{1}{c|}{$1.25\times 10^{-30}$} & $1.58\times 10^{-38}$ \\ \hline
$0.2$                & \multicolumn{1}{c|}{$8.70\times 10^{-15}$} & \multicolumn{1}{c|}{$1.48\times 10^{-23}$} & \multicolumn{1}{c|}{$4.22\times 10^{-32}$} & $1.50\times 10^{-40}$ \\ \hline
$0.3$                & \multicolumn{1}{c|}{$3.94\times 10^{-15}$} & \multicolumn{1}{c|}{$2.41\times 10^{-24}$} & \multicolumn{1}{c|}{$2.52\times 10^{-33}$} & $3.28\times 10^{-42}$ \\ \hline
$0.4$                & \multicolumn{1}{c|}{$1.91\times 10^{-15}$} & \multicolumn{1}{c|}{$4.03\times 10^{-25}$} & \multicolumn{1}{c|}{$1.70\times 10^{-34}$} & $8.63\times 10^{-44}$ \\ \hline
$0.5$                & \multicolumn{1}{c|}{$1.37\times 10^{-15}$} & \multicolumn{1}{c|}{$2.61\times 10^{-27}$} & \multicolumn{1}{c|}{$2.05\times 10^{-35}$} & $7.30\times 10^{-47}$ \\ \hline
$0.6$                & \multicolumn{1}{c|}{$2.05\times 10^{-15}$} & \multicolumn{1}{c|}{$4.53\times 10^{-25}$} & \multicolumn{1}{c|}{$1.96\times 10^{-34}$} & $1.04\times 10^{-43}$ \\ \hline
$0.7$                & \multicolumn{1}{c|}{$4.44\times 10^{-15}$} & \multicolumn{1}{c|}{$2.93\times 10^{-24}$} & \multicolumn{1}{c|}{$3.30\times 10^{-33}$} & $4.64\times 10^{-42}$ \\ \hline
$0.8$                & \multicolumn{1}{c|}{$1.03\times 10^{-14}$} & \multicolumn{1}{c|}{$1.98\times 10^{-23}$} & \multicolumn{1}{c|}{$6.41\times 10^{-32}$} & $2.59\times 10^{-40}$ \\ \hline
$0.9$                & \multicolumn{1}{c|}{$2.58\times 10^{-14}$} & \multicolumn{1}{c|}{$1.90\times 10^{-22}$} & \multicolumn{1}{c|}{$2.37\times 10^{-30}$} & $3.68\times 10^{-38}$ \\ \hline
\end{tabular}
\medskip

\begin{tabular}{|c|cccc|}
\hline
\multirow{2}{*}{$y$} & \multicolumn{4}{c|}{ $|x_{1}-\alpha|$ for $H_{2}(\alpha)=y$}                                                                                                                                                           \\ \cline{2-5} 
                     & \multicolumn{1}{c|}{$5$}                                                      & \multicolumn{1}{c|}{$10$}                  & \multicolumn{1}{c|}{$15$}                  & $20$                  \\ \hline
$0.1$                & \multicolumn{1}{c|}{$1.97\times 10^{-14}$}                                    & \multicolumn{1}{c|}{$1.97\times 10^{-22}$} & \multicolumn{1}{c|}{$1.22\times 10^{-30}$} & $1.55\times 10^{-38}$ \\ \hline
$0.2$                & \multicolumn{1}{c|}{$8.59\times 10^{-15}$}                                    & \multicolumn{1}{c|}{$1.46\times 10^{-23}$} & \multicolumn{1}{c|}{$4.16\times 10^{-32}$} & $1.48\times 10^{-40}$ \\ \hline
$0.3$                & \multicolumn{1}{c|}{$3.91\times 10^{-15}$}                                    & \multicolumn{1}{c|}{$2.40\times 10^{-24}$} & \multicolumn{1}{c|}{$2.50\times 10^{-33}$} & $3.26\times 10^{-42}$ \\ \hline
$0.4$                & \multicolumn{1}{c|}{$1.90\times 10^{-15}$}                                    & \multicolumn{1}{c|}{$4.02\times 10^{-25}$} & \multicolumn{1}{c|}{$1.68\times 10^{-34}$} & $8.60\times 10^{-44}$ \\ \hline
$0.5$                & \multicolumn{1}{c|}{$1.37\times 10^{-15}$}                                    & \multicolumn{1}{c|}{$2.61\times 10^{-27}$} & \multicolumn{1}{c|}{$2.05\times 10^{-35}$} & $7.30\times 10^{-47}$ \\ \hline
$0.6$                & \multicolumn{1}{c|}{$2.06\times 10^{-15}$}                                    & \multicolumn{1}{c|}{$4.54\times 10^{-25}$} & \multicolumn{1}{c|}{$1.96\times 10^{-34}$} & $1.04\times 10^{-43}$ \\ \hline
$0.7$                & \multicolumn{1}{c|}{$4.47\times 10^{-15}$}                                    & \multicolumn{1}{c|}{$2.96\times 10^{-24}$} & \multicolumn{1}{c|}{$3.33\times 10^{-33}$} & $4.68\times 10^{-42}$ \\ \hline
$0.8$                & \multicolumn{1}{c|}{$1.03\times 10^{-14}$} & \multicolumn{1}{c|}{$2.01\times 10^{-23}$} & \multicolumn{1}{c|}{$6.49\times 10^{-32}$} & $2.62\times 10^{-40}$ \\ \hline
$0.9$                & \multicolumn{1}{c|}{$2.62\times 10^{-14}$}                                    & \multicolumn{1}{c|}{$1.94\times 10^{-22}$} & \multicolumn{1}{c|}{$2.42\times 10^{-30}$} & $3.75\times 10^{-38}$ \\ \hline
Theorem  \ref{maintheorem}            & \multicolumn{1}{c|}{3 s}                                                    & \multicolumn{1}{c|}{4 s}                 & \multicolumn{1}{c|}{6 s}                 & 8 s                 \\ \hline
Householder's method & \multicolumn{1}{c|}{13 s}                                                   & \multicolumn{1}{c|}{63 s}           & \multicolumn{1}{c|}{194 s}          & 543 s           \\ \hline
\end{tabular}
\medskip
\caption{A table of difference between $x_{1}$ (value after 1 iteration) and the actual root $\alpha$ for $0.1\leq y\leq 0.9$, and for convergence orders $5$, $10$, $15$, $20$. Here we use the function $x$ to approximate $H_{q}(x)$. A comparison of running times in seconds between Theorem \ref{maintheorem} and Householder's method is also shown.}\label{Hq-x}
\end{table}

\section{Application 3: Basins of Attraction}

One of the things to study about root finding methods is their basins of attraction. Suppose we have an equation $f(x)=0$ whose solution set is discrete $\{\zeta_{i}:\ \zeta_{i}\in\mathbb{N}\}$. As the root finding method involves finding a starting point $x_{0}$, the iteration provides a sequence $\{x_{i}\}_{i=0}^{\infty}$ that either converges to one of the solutions $\zeta_{k}$, or it diverges. These starting points $x_{0}$ are then grouped together according to their corresponding limit values, and they form partitions on the domain of the function $f$. Having explained the idea behind the basins of attractions, we recall the following definition (cf. \cite{Epperson-2015, NRA-2020, OR-2000}).
\medskip
\begin{definition}
Given a function $f$ defined on the complex plane $\mathbb{C}$, with roots $\zeta_{1}$, $\zeta_{2}$, $\cdots$, $\zeta_{k}$, and a (convergent) root-finding iteration given by 
$
z_{n+1} = g(z_{n}),
$
the basin of attraction for the root $\zeta_{k}$ is defined to be 
\[
B_{f,g}(\zeta_{k})=
\left\{
\zeta \in\mathbb{C}:\ \text{the iteration }z_{n+1}=g(z_{n})\text{ with }z_{0}=\zeta\text{ converges to }\zeta_{k}\right\}.
\]
\end{definition}

We first examine the case over real functions, with focus on the real roots. We adapt the definition to the real case. 
\medskip
\begin{definition}
Let $U\subset\mathbb{R}$ be a subset and $f:U\rightarrow\mathbb{R}$ a continuous function with roots $\zeta_{1}$, $\cdots$, $\zeta_{k}$ in $U$. Given a (convergent) root-finding iteration  
$
x_{n+1}=g(x_{n}),
$
the basin of attraction for the root $\zeta_{k}$ is defined to be
\[
B_{f,g}^{\mathbb{R}}(\zeta_{k})=\left\{
\zeta \in\mathbb{R}:\ \text{the iteration }x_{n+1}=g(x_{n})\text{ with }x_{0}=\zeta\text{ converges to }\zeta_{k}\right\}.
\]
\end{definition}

Just as in the case of Householder's method, the root-finding method in Theorem \ref{maintheorem} can be applied directly to complex equations for complex solutions. As a result of Theorem \ref{Theorem-Householder}, the  following corollary is trivial
\medskip
\begin{corollary} \label{cor-BA}
	Let $B_{f,g}^{\text{H}}(\zeta_{k})$ be the basin of attraction for the root $\zeta_{k}$ using Householder's method, and let $B_{f,g}(\zeta_{k})$ be the basin of attraction for the root $\zeta_{k}$ using Theorem \ref{maintheorem}. Then $B_{f,g}^{\text{H}}(\zeta_{k})=B_{f,g}(\zeta_{k})$. The same conclusion holds for the real case. 
\end{corollary}
\medskip

As an example to Corollary \ref{cor-BA}, we compute basins of attractions for the equation $z^{6}-z+1=0$ and compare them across various orders.  In general, we expect them to exhibit fractal behaviour with six large components corresponding to the degree of the polynomial. We see that in Figure \ref{BA-2D}, these components should then become clearer and more distinct with less fractal behaviour as the order of root-finding method increases. More importantly, both the 6-th order Householder's method and the proposed method in Theorem \ref{maintheorem} have the same basins of attraction.

While the examples deal with holomorphic polynomials, finding initial guesses for root-finding method to solve for an equation involving a holomorphic function may be less straightforward, since we are dealing with infinite series in general. Nonetheless, there is an analogue of the Stone-Weierstrass theorem for complex analytic setting that allows us to uniformly approximate holomorphic functions with holomorphic polynomials on some compact sets.

\begin{theorem}[Mergelyan Theorem]
	Let $K$ be a compact subset of $\mathbb{C}$ such that $\mathbb{C}-K$ is connected. Then every continuous function $f:K\rightarrow \mathbb{C}$ such that the restriction $f$ to $\text{int}(K)$ is holomorphic can be approximated uniformly on $K$ with holomorphic polynomials. 
\end{theorem}

In the future work, we will study  the construction of such polynomials and the application of the proposed root-finding method to solve  the roots of analytic functions.

\begin{figure}[th!]
	\centering
\includegraphics[scale=0.4]{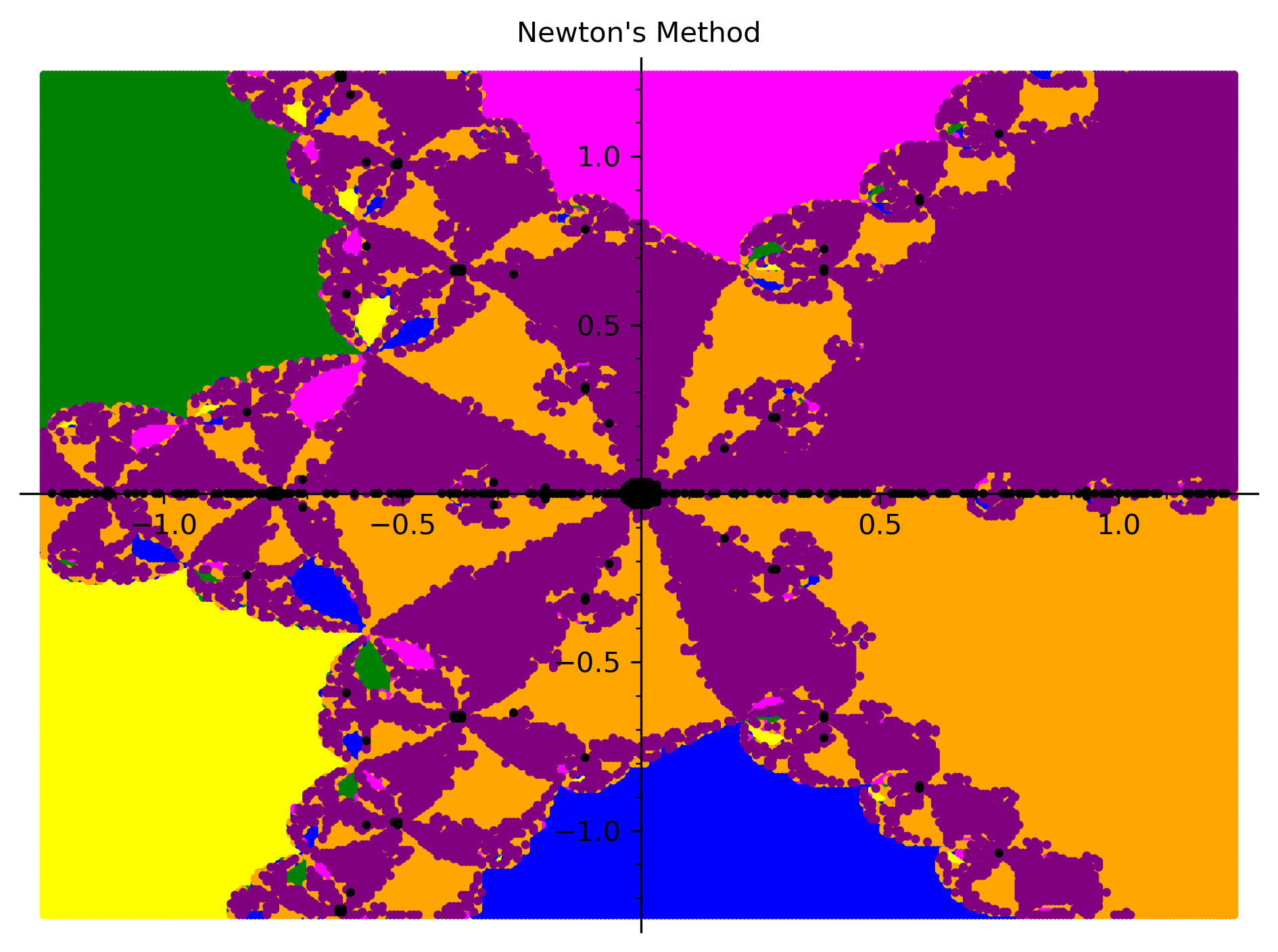}
\includegraphics[scale=0.4]{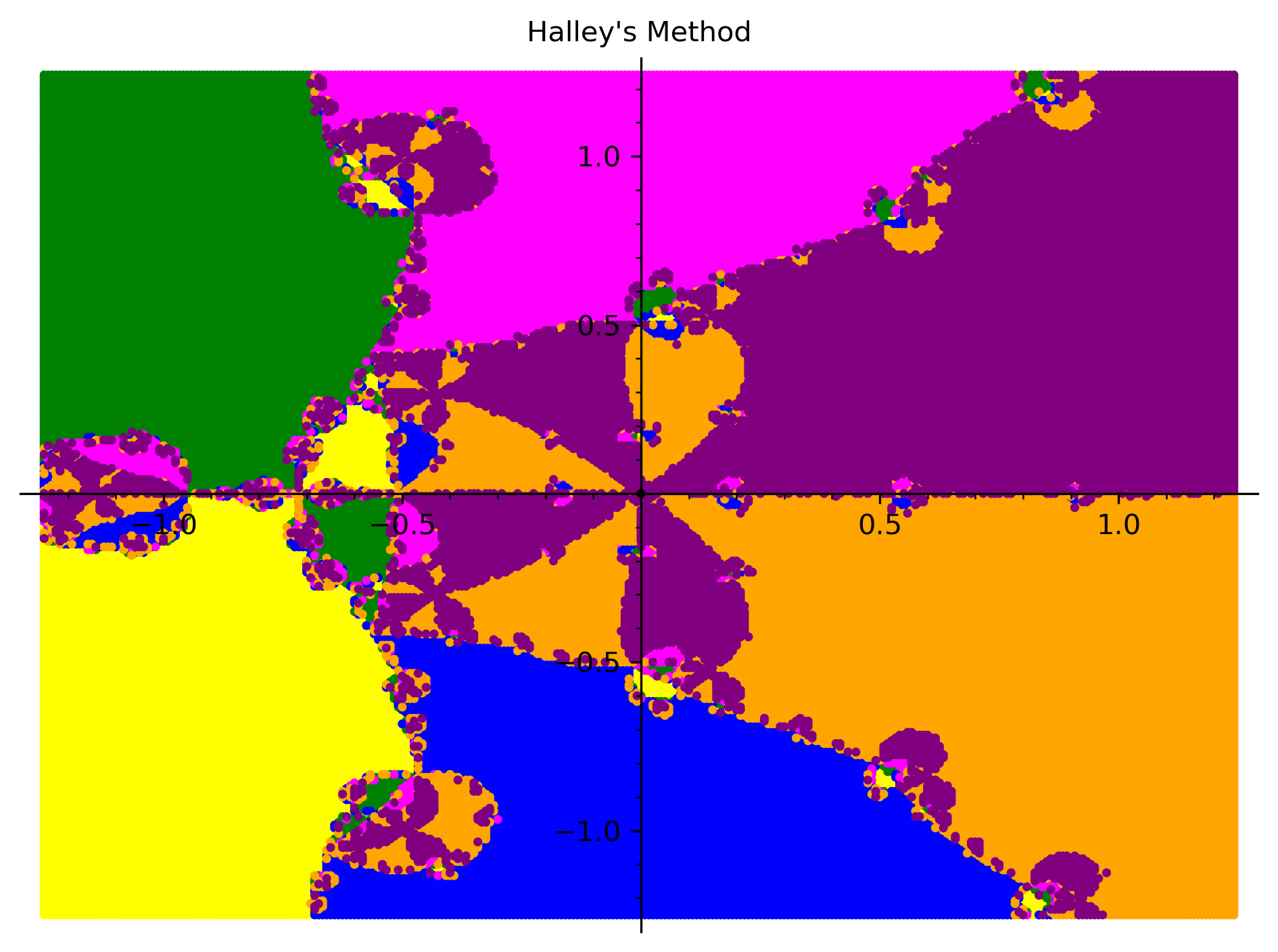}

\includegraphics[scale=0.4]{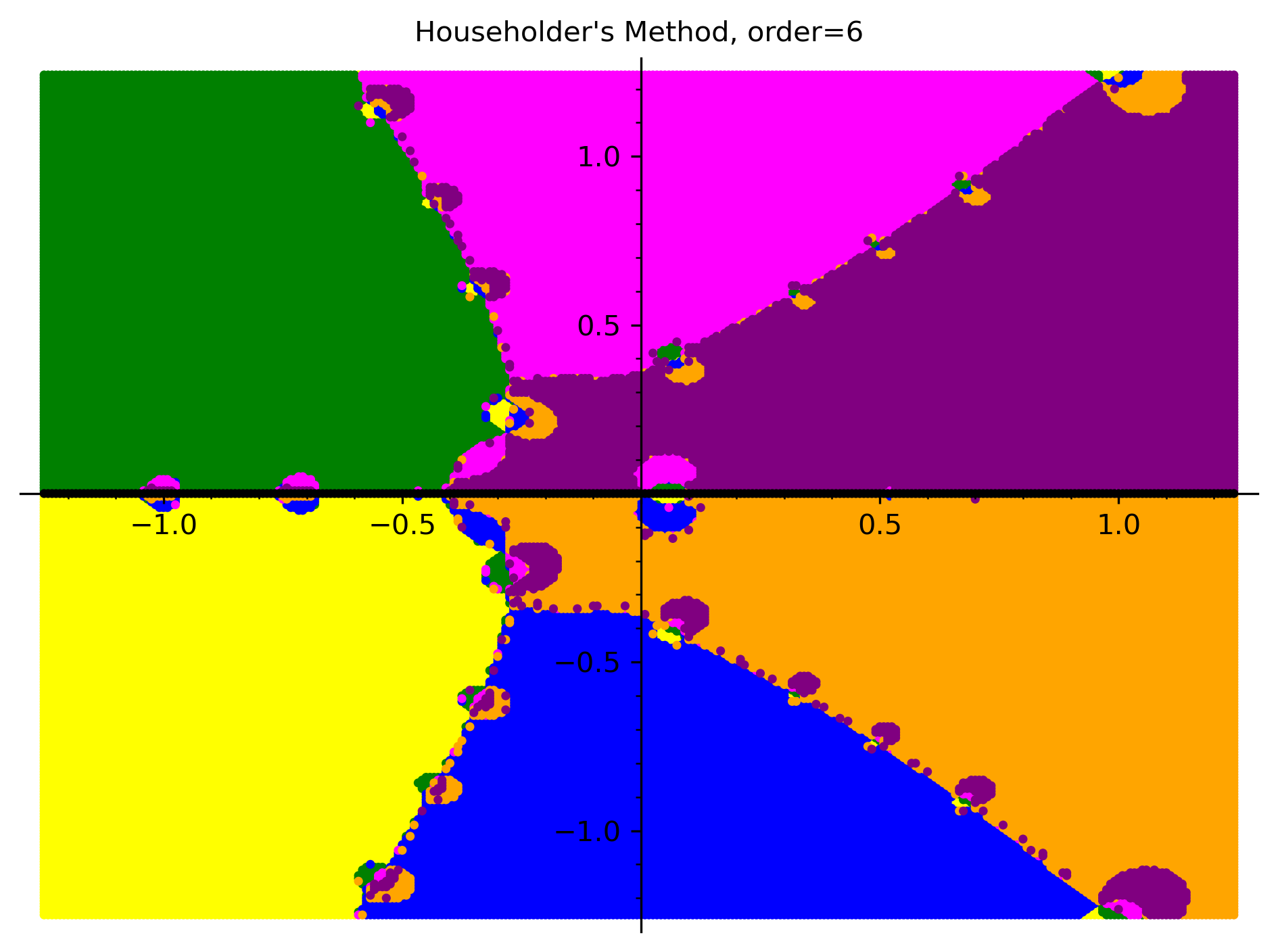}
\includegraphics[scale=0.4]{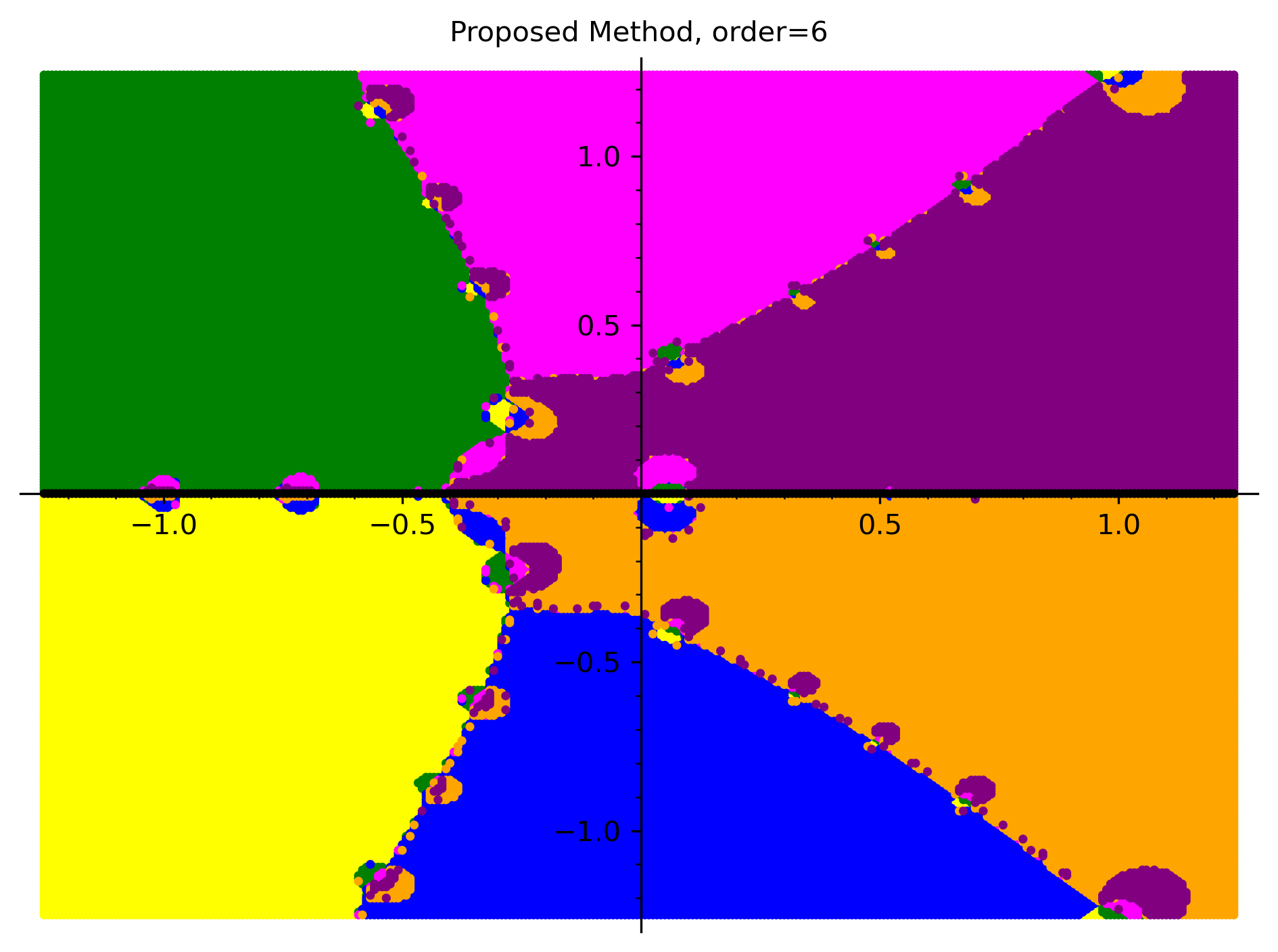}
\caption{Basins of attraction for various root finding methods for the equation $z^{6}-z+1=0$.}\label{BA-2D}
\end{figure}

\section{Conclusion}

We have proposed a root-finding scheme that allows finding numerical solutions to $f(x)=0$ at high order. This is possible because the scheme sets up a system of linear equations and performs eliminations of coefficients similar to the Gaussian elimination method. As an application, we show that Theorem \ref{maintheorem} can be used to numerically compute Householder's method, with a more precise computation of the convergence factor $C$. We apply Theorem \ref{maintheorem} to find the inverses of the $q$-ary entropy function, and we show that it can run much faster than Householder's method. Finally, we show that both Theorem \ref{maintheorem} and Householder's method have the same basins of attraction. 

One of the main advantages of Theorem \ref{maintheorem} over the Householder's method is that it relies much less on symbolic calculations. While Householder's method requires taking higher order derivatives of the reciprocal of a function, Algorithm \ref{RFA} relies only on the Taylor series expansion of a function at the $n$-th iteration $x_{n}$. As a result, we can apply Theorem \ref{maintheorem} to compute basins of attraction even for non-polynomial equations, with higher-order root-finding method. Indeed Figure \ref{BA1d} shows the basins of attraction of the equation $H_{2}(x)-0.3=0$ for order $50$ and for 3000 points between $0\leq x\leq 1$ after 1 iteration. On the other hand, Householder's method takes a very long time to run. Our future works can be applied to the study of dynamical behaviour of the root-finding methods for non-algebraic equations.
 
\begin{figure}
\center \includegraphics[scale=0.50]{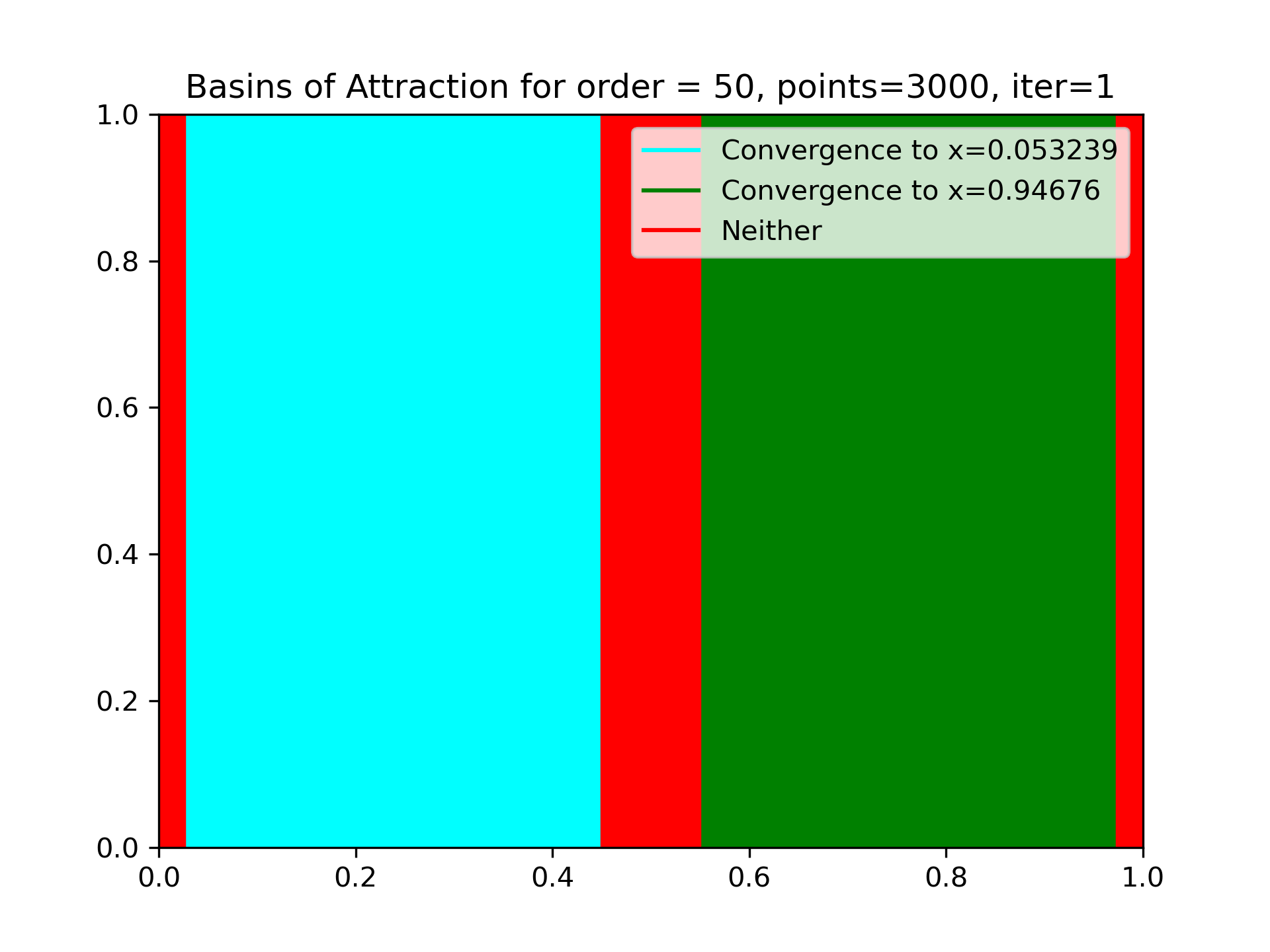}
\caption{Basins of Attraction for 50-th order root-finding method for 3000 points after 1 iteration for the equation $H_{2}(x)=0.3$ using Theorem \ref{maintheorem}.}\label{BA1d}
\end{figure}

\section{Disclosure statement}
The authors report there are no competing interests to declare.

\bibliographystyle{plain} 
\bibliography{bibref} 
\end{document}